\documentclass[11pt,a4paper]{article}

\usepackage[utf8]{inputenc}
\usepackage{amsmath}
\usepackage{amssymb}
\usepackage{amsthm}
\usepackage{hyperref}

\usepackage{orcidlink}
\usepackage{authblk}

\usepackage{tikz} 
\usetikzlibrary{arrows.meta, positioning}

\newcommand{\N}{\mathbb{N}}
\newcommand{\R}{\mathbb{R}}

\def\sub{\subseteq}

\def\into{\longrightarrow}
\def\setdiff{\backslash}
\def\nseqLim{{\displaystyle\lim_{n\rightarrow\infty}}}
\def\kseqLim{{\displaystyle\lim_{k\rightarrow\infty}}}

\def\theSet{\mathcal{S}}

\def\implies{\Longrightarrow}

 \newtheorem{theorem}{Theorem}[section]
 
  \newtheorem{proposition}[theorem]{Proposition}
 
 \newtheorem{remark}[theorem]{Remark}

\theoremstyle{definition}
 \newtheorem{example}[theorem]{Example}

 \def\AP{\ensuremath{{\bf(AP)}}}
 \def\NIP{\ensuremath{{\bf(NIP)}}}
 
 \def\ES{\ensuremath{{\bf(ES)}}}
 \def\CA{\ensuremath{{\bf(CA)}}}
 
 \def\BWP{\ensuremath{{\bf(BWP)}}}

 \def\BVT{\ensuremath{{\bf(BVT)}}}
 \def\IVT{\ensuremath{{\bf(IVT)}}}
 \def\DIT{\ensuremath{{\bf(DIT1)}}}
 \def\EVT{\ensuremath{{\bf(EVT)}}}
 \def\RT{\ensuremath{{\bf(RT)}}}
 \def\eMVT{\ensuremath{{\bf(eMVT)}}}
 \def\MVT{\ensuremath{{\bf(MVT)}}}
 \def\SIFT{\ensuremath{{\bf(SIFT)}}}
 \def\IFT{\ensuremath{{\bf(IFT)}}}
 \def\CFT{\ensuremath{{\bf(CFT)}}}
 \def\IET{\ensuremath{{\bf(DIT2)}}}
 \def\DT{\ensuremath{{\bf(DT)}}}
 \def\iFTC{\ensuremath{{\bf(FTC1)}}}
 \def\iiFTC{\ensuremath{{\bf(FTC2)}}}
 \def\Star{\ensuremath{{\bf(\star)}}}

\title{Naive Real Function Theory}
\author{Rafael Reno S. Cantuba\thanks{Associate Professor, Department of Mathematics and Statistics, De La Salle University (DLSU), Taft Ave., Manila, Philippines, email: rafael.cantuba@dlsu.edu.ph}\\ ORCID: 0000-0002-4685-8761\orcidlink{0000-0002-4685-8761}}
\date{}

\begin{document}

\maketitle

\begin{abstract} A naive treatment of real function theory is given with the choice of completeness axiom as one of: the nested intervals principle (in conjunction with the Archimedean property), the Dedekind cut axiom, the existence of suprema, and the Bolzano-Weierstrass property of sequences. These four statements are proven at the beginning to be equivalent. The intermediate value property and integrability of functions continuous on a closed and bounded interval are then proven using four types of compactness arguments, each corresponding to one of the aforementioned completeness axioms. From integrability properties, the so-called geometric principles of calculus, and the fundamental theorems of calculus are proven. All theorems that are terminal in the chain of implications are then proven to be sufficient conditions for the Dedekind cut axiom.
\end{abstract}

\vspace{0.5em}
\noindent
\begin{quote}
    
    \vspace{0.5em}
\end{quote}

\section{Introduction} A mathematical interlude concerning differential calculus may be found in the second lecture or chapter of the physics book \cite{sus13}. The interlude begins with the notion of limit\textemdash of sequences, of functions at a point, and of a difference quotient, leading to the derivative\textemdash but without epsilon-delta definitions. The power formula for differentiation is then derived. The derivatives of $\sin$, $\cos$, $\exp$ and $\ln$ are introduced without proof, followed by the constant rule, linearity rules, product rule and chain rule for differentiation. The interlude ends with: ``Using these rules, you can calculate a lot of derivatives. That's basically all there is to differential calculus." 

While such a computational summary may suffice for the immediate practical demands of an introductory physics course, it leaves out the deeper mathematical substance that gives calculus its coherence and rigor. Reducing differential calculus to a list of differentiation rules obscures the foundational theorems that justify those very rules and explain what derivatives actually tell us about functions. More broadly, this perspective overlooks the entire logical architecture of the underlying mathematical theory, in which the properties of continuous, differentiable, and integrable functions are systematically deduced from the completeness of the real numbers.

The richness, and consequently the beauty, of the mathematical theory behind differentiation and integration, the two main processes of calculus, rests on two fundamental notions. First is that some principles from calculus are suggested to us by geometric intuition and aid us in understanding, using calculus, the ``shape'' of a graph of a function. Second, some principles serve as the ``workhorses of calculus'' \cite[p.~107]{roy10}, in the sense that these theorems transform for us the abstract and difficult computations of integrals or derivatives into straightforward algebraic evaluations, and this is the theoretical foundation of differentiation formulae and integral formulae. Thus, there is more to differential calculus (and integral calculus) than just these formulae. 

The theorems that fall under the first class shall be referred to here as the ``geometric principles of calculus,'' a term we took from the classic book \cite{loo74}. More specifically, in \cite[Section~11.7]{loo74}, Darboux's Theorem, the Strictly Increasing Function Theorem, the Constant Function Theorem, the Mean Value Theorem, the Extended Mean Value Theorem, the Intermediate Value Theorem and the Extreme Value Theorem are presented as geometric principles that help us understand properties of graphs of functions based on differentiability and continuity, and their logical interconnections as analytic theorems were developed. The second class of theorems include the two forms of the fundamental theorem of calculus.

This brings us to real analysis or the theory behind why calculus may be done on real-valued functions of a real variable. The more precise but less popular term is real function theory. A theory in mathematics is simply a collection of theorems, including all the definitions and proofs, but more importantly, the starting point or points of the theory: the statements accepted without proof, which are the axiom or axioms of the theory. For real function theory, aside from the ordered field axioms for the real numbers and the relevant principles from set theory, the most important postulate is the completeness axiom. The completeness axiom has many forms, as evidenced by the emergence of studies like \cite{can24,dev14,pro13,rie01,tei13}. Our goal in this exposition is to give a naive treatment of real function theory in the sense that most of the axiomatics of ordered fields and set theory are taken for granted so that a student who knows calculus and basic proving techniques may experience some expediency in studying real function theory. The treatment must also be comprehensive enough to cover the aforementioned geometric principles and fundamental theorems of calculus with the minimum possible amount of notions and preliminaries from standard real analysis. All such preliminaries are discussed in Section~\ref{PrelSec}, which, for instance, do not include limit superior, Cauchy sequences, the monotone convergence principle, open sets, closed sets, among others. It seemed that only those in Section~\ref{PrelSec} are enough to cover the geometric principles and fundamental theorems, and the theory culminates in Section~\ref{RealSumSec}, where a diagram of the chains of implications is shown. The section that follows, Section~\ref{DedekindSec}, is only for an additional theoretical purpose of showing which of the theorems are equivalent to a completeness axiom. It is only here that open and closed sets of real numbers are needed for a more organized presentation, but the aforementioned essentials of real function theory have been developed without requiring the student a serious study of the standard topology for the real number line.

The amount of preliminaries and implications that need to be proven may still be further reduced if one is to choose which completeness axiom is best desired for a naive development of real function theory. The reader has four choices: the nested intervals principle (in conjunction with the Archimedean property), the Dedekind cut axiom, the existence of suprema, and the Bolzano-Weierstrass property of sequences. These four are proven to be equivalent in Section~\ref{AxiomSec}, so just one of these four completeness axioms may serve as a starting point of the theory, and each axiom corresponds to a certain type of argument that will be used to prove the first theorems (which are the Intermediate Value Theorem and two integration theorems; the rest of the theory follows). We shall call this argument a ``compactness argument,'' a term we borrowed from \cite[Section~2]{tho07}. Even if we shall not cover compactness (topology), we decided to retain the term.

\section{Preliminaries}\label{PrelSec} This exposition is written under the writing philosophy that excessive citation to theorem numbers negatively affects the readability of a real analysis proof. Thus, the definitions and theorems that serve as our preliminaries are not numbered, or equivalently, interspersed with the prose. Nevertheless, effort has been made to make sure that all properties or statements mentioned without proof in Sections~\ref{AxiomSec}--\ref{iiFTCsec} may be found in this section. Even also without proofs, we hope this section serves as a good guide for the reader on what notions from real analysis to review.

\paragraph{Some properties of real numbers.} The set of all real numbers shall be denoted by $\R$, and the set of all positive integers, or equivalently, the set of all natural numbers shall be denoted by $\N$. We assume familiarity with the field and order axioms of $\R$, routine manipulation of inequalities, and the properties of the absolute value $|x|$, including the triangle inequality $|x-y|\leq |x-z|+|z-y|$ for all $x,y,z\in\R$. We also assume the Principle of Mathematical Induction and the Well-Ordering Principle, which asserts that every nonempty subset of $\N$ has a least element. By mathematical induction, $n<2^n$ for all $n\in\N$. Furthermore, $\N$ is an infinite set, any subset of $\R$ containing $\N$ is infinite, and removing a finite subset $\{1,2,\ldots,N\}$ from an infinite subset of $\N$ leaves an infinite, and hence nonempty, subset. Given $a,b\in\R$ with $a<b$, we use standard notation for the \emph{closed and bounded interval} $[a,b]$, the \emph{open and bounded interval} $(a,b)$, the \emph{unbounded open intervals} $(-\infty,a)$, $(b,\infty)$, and the \emph{unbounded closed intervals}  $(-\infty,a]$, $[b,\infty)$.

\paragraph{Dedekind cuts.} A pair $(A,B)$ of subsets of $\R$ is a \emph{(Dedekind) cut} of $\R$ if $A$ and $B$ are nonempty, $A\cap B=\emptyset$, $A\cup B=\R$, and $a<b$ for all $a\in A$ and $b\in B$. If $a\in A$, then any $b\in B$ satisfies $a<b$ and hence is not an element of $(-\infty,a]$, which is hence disjoint from $B$ and is consequently a subset of $B^c=A$. That is, $a\in A$ implies $(-\infty,a]\sub A$. By an analogous argument, $b\in B$ implies $[b,\infty)\sub B$. We may refer to this property of a cut $(A,B)$ as $A$ being \emph{downward closed} and $B$ as \emph{upward closed}. These should not be confused with closed sets in the standard topology for the real numbers.

A real number $c$ is a \emph{cut point} of a cut $(A,B)$ if $a\leq c\leq b$ for all $a\in A$ and $b\in B$. As an immediate consequence, a cut of $\R$ can have at most one cut point. Moreover, the defining condition for a cut point $c$ of a cut $(A,B)$ may also be written as $A\sub(-\infty,c]$ and $B\sub[c,\infty)$. Taking complements, this is equivalent to $(c,\infty)\sub A^c=B$ and $(-\infty,c)\sub B^c=A$. 

\paragraph{Bounds, suprema and infima.} Given a nonempty subset $\theSet$ of $\R$, a real number $b$ is an \emph{upper bound} of $\theSet$ if $\theSet\sub(-\infty,b]$, and a real number $a$ is a \emph{lower bound} of $\theSet$ if $\theSet\sub[a,\infty)$. An upper bound $M$ of $\theSet$ is the \emph{supremum} of $\theSet$ (written $M=\sup\theSet$) if all upper bounds of $\theSet$ are in the interval $[M,\infty)$. Dually, a lower bound $m$ of $\theSet$ is the \emph{infimum} of $\theSet$ (written $m=\inf\theSet$) if all lower bounds of $\theSet$ are in the interval $(-\infty,m]$. Directly from the definitions and routine manipulation of inequalities, the supremum and infimum of $\theSet$, when they exist, are unique; a real number that is both an element and an upper bound of $\theSet$ is equal to $\sup\theSet$; if $a$ is a lower bound and $b$ is an upper bound of $\theSet$, then $a\leq \inf\theSet\leq\sup\theSet\leq b$; and if $y\in\R$ satisfies $y<\sup\theSet$, then $y$ is not an upper bound of $\theSet$, so there exists $s\in\theSet$ such that $y<s\leq\sup\theSet$ (with the analogous property holding for $\inf\theSet$). When $\theSet$ is a nonempty subset of $\R$ with no upper bound in $\R$, we write $\sup\theSet=\infty$, and when $\theSet$ has no lower bound in $\R$, we write $\inf\theSet=-\infty$; thus, $\sup\theSet\in\R\cup\{\infty\}$ and $\inf\theSet\in\R\cup\{-\infty\}$ for every nonempty subset $\theSet$ of $\R$. We equip $\R\cup\{-\infty,\infty\}$ with the standard extended real order and arithmetic conventions: $-\infty<x<\infty$ for all $x\in\R$; $-\infty<\infty$; $\infty\cdot c=\infty$ and $(-\infty)\cdot c=-\infty$ for every real number $c>0$; $\infty+x=\infty$ for all $x\in\R\cup\{\infty\}$; $-\infty+x=-\infty$ for all $x\in\R\cup\{-\infty\}$; $\infty-x=\infty$ for all $x\in\R\cup\{-\infty\}$; and $y-(-\infty)=\infty$ for all $y\in\R\cup\{\infty\}$. Under these rules, if $U,L\in\R\cup\{-\infty,\infty\}$ such that $0\leq U-L<\varepsilon$ for some $\varepsilon\in\R$, then $U,L\in\R$.

\paragraph{Sequences.} A \emph{sequence} in a set $X$ is a function $\N\into X$, which we denote by $(a_n)$, where $a_n\in X$ is the value at $n\in\N$. A sequence $(a_n)$ in $\R$ is \emph{bounded} if there exists a closed and bounded interval $[a,b]$ such that $(a_n)$ is a sequence in $[a,b]$. A sequence $(a_n)$ in $\R$ is \emph{monotonically increasing} if $a_n\leq a_{n+1}$ for all $n\in\N$, and \emph{strictly increasing} if $a_n<a_{n+1}$ for all $n\in\N$. A sequence $(a_n)$ in $\R$ \emph{converges} to $L\in\R$ (or has \emph{limit} $L$, written $\nseqLim a_n=L$) if for every $\varepsilon>0$, there exists $N\in\N$ such that for all $n\in\N$, if $n\geq N$, then $|a_n-L|<\varepsilon$. In such a case, we say that $(a_n)$ is \emph{convergent}. By standard epsilon arguments for sequence limits, one establishes the following properties: the limit of a convergent sequence is unique; if $h,k\in\R$ and if $(a_n)$ and $(b_n)$ are convergent sequences in $\R$, then $(a_n+b_n)$ and $(ha_n+kb_n)$ are convergent with $\nseqLim(a_n+b_n)=\nseqLim a_n+\nseqLim b_n$ and $\nseqLim(ha_n+kb_n)=h\nseqLim a_n+k\nseqLim b_n$ (and these are called addition and linearity rules for sequence limits, respectively); if $(a_n)$ and $(b_n)$ are convergent sequences satisfying $a_n\leq b_n$ for all $n\in\N$, then $\nseqLim a_n\leq\nseqLim b_n$ (the monotonicity of sequence limits); and if $(a_n)$, $(b_n)$, and $(c_n)$ are sequences in $\R$ satisfying $a_n\leq c_n\leq b_n$ for all $n\in\N$ with $\nseqLim a_n=L=\nseqLim b_n$, then $(c_n)$ is convergent with $\nseqLim c_n=L$ (the Squeeze Theorem for sequence limits). In most applications, the sequence $(a_n)$ in this statement of the Squeeze Theorem is usually taken as a constant sequence, which converges to the underlying constant.

A sequence $(b_k)$ is a \emph{subsequence} of $(a_n)$ if there exists a strictly increasing sequence $(N_k)$ in $\N$ such that $b_k=a_{N_k}$ for all $k\in\N$, in which case we denote the subsequence by $(a_{N_k})$. By mathematical induction, the indices of any subsequence $(a_{N_k})$ satisfy $k\leq N_k$ for all $k\in\N$, and $m\leq h$ implies $N_m\leq N_h$. Furthermore, by an epsilon argument for sequence limits, if $(a_n)$ converges to $L\in\R$, then every subsequence $(a_{N_k})$ also converges to $L$.

 A sequence $([a_n,b_n])$ of closed and bounded intervals is \emph{nested} if\linebreak $[a_{n+1},b_{n+1}]\sub[a_n,b_n]$ for all $n\in\N$. Directly from the definition, if $([a_n,b_n])$ is nested, then $a_1\leq a_n\leq a_{n+1}\leq b_{n+1}\leq b_n\leq b_1$ for all $n\in\N$, and consequently $a_n\leq b_m$ for all $m,n\in\N$.

\paragraph{Continuity.} Given a function $f:X\into\R$, given $x_1,x_2\in X$, the condition ``$x_1=x_2$ implies $f(x_1)=f(x_2)$,'' is always true. For any subset $E\sub X$, the set $\{f(x)\,:\,x\in E\}$ is the \emph{image} of $E$ under $f$. A function $f:X\into\R$ is \emph{continuous at} $c\in X$ if for every $\varepsilon>0$, there exists $\delta>0$ such that $x\in X$ and $|x-c|<\delta$ imply $|f(x)-f(c)|<\varepsilon$. We say that $f$ is \emph{continuous on} $E\sub X$ if $f$ is continuous at every $c\in E$. A function $f:X\into\R$ has \emph{constant sign} on $E\sub X$ if either $\{f(x)\,:\,x\in E\}\sub(0,\infty)$ or $\{f(x)\,:\,x\in E\}\sub(-\infty,0)$. Directly from the definition of continuity, if $f$ is continuous on $X$ and $E\sub X$, then $f$ is continuous on $E$. 

Using epsilon-delta arguments (and epsilon arguments for sequence limits where appropriate), one obtains the following properties of continuous functions:

 If $f:X\into\R$ is continuous at $c\in X$ and $(x_n)$ is a sequence in $X$ converging to $c$, then $\nseqLim f(x_n)=f(c)$.

    If $f:[a,b]\into\R$ is continuous at $c\in[a,b]$ and $f(c)\neq 0$, then there exists $\delta>0$ such that $f$ has constant sign (matching the sign of $f(c)$) on $\{x\in[a,b]\,:\,|x-c|<\delta\}$; in particular, if $c=a$ (respectively, $c=b$), there exists $\delta_1>0$ (respectively, $\delta_2>0$) such that $f$ has constant sign on $[a,a+\delta_1]\sub[a,b]$ (respectively, $[b-\delta_2,b]\sub[a,b]$), and if $c\in(a,b]$, there exists $\delta_3>0$ such that $a<a+\delta_3<c$ and $f$ has constant sign on $[a+\delta_3,c]$.
  
    If $f,g:X\into\R$ are continuous on $X$ and $c_1,c_2\in\R$, then $c_1f+c_2g$ is continuous on $X$, and if $g(x)\neq 0$ for all $x\in X$, then the function $x\mapsto\frac{1}{g(x)}$ is continuous on $X$. Constant and linear functions (or at most degree one polynomial functions) are continuous.

    If $f:[a,b]\into\R$ is continuous on $[a,b]$, then for every $\varepsilon>0$, there exists $\delta>0$ such that for all $x,t\in[a,b]$, if $|x-t|<\delta$ then $|f(x)-f(t)|<\varepsilon$.

\paragraph{Integration.} Given a closed and bounded interval $[a,b]$, a finite subset $\Delta$ of $[a,b]$ is a \emph{partition} of $[a,b]$ if $\{a,b\}\sub\Delta$. We denote the collection of all partitions of $[a,b]$ by $\mathcal{P}[a,b]$. Listing the elements of $\Delta\in\mathcal{P}[a,b]$ in strictly increasing order as $a=x_0<x_1<\cdots<x_n=b$, for any function $f:[a,b]\into\R$, the \emph{upper Darboux sum} $U_f(\Delta)$ and \emph{lower Darboux sum} $L_f(\Delta)$ of $f$ with respect to $\Delta$ are defined by
\begin{eqnarray}
U_f(\Delta):=\sum_{k=1}^n\left(\sup_{x\in[x_{k-1},x_k]}f(x)\right)(x_k-x_{k-1}),\nonumber\\
L_f(\Delta):=\sum_{k=1}^n\left(\inf_{x\in[x_{k-1},x_k]}f(x)\right)(x_k-x_{k-1}),\nonumber
\end{eqnarray}
respectively. Since each subinterval $[x_{k-1},x_k]$ is nonempty and since\linebreak $x_k-x_{k-1}>0$, the rules on the arithmetic of extended real numbers ensure that $U_f(\Delta)\in\R\cup\{\infty\}$, $L_f(\Delta)\in\R\cup\{-\infty\}$, and $U_f(\Delta)-L_f(\Delta)\in\R\cup\{\infty\}$, with $U_f(\Delta)-L_f(\Delta)=\infty$ whenever $U_f(\Delta)=\infty$ or $L_f(\Delta)=-\infty$. Given $\Delta_1,\Delta_2\in\mathcal{P}[a,b]$, we say that $\Delta_2$ is a \emph{refinement} of $\Delta_1$ if $\Delta_1\sub\Delta_2$. Directly from the definitions and routine manipulation of inequalities and finite sums, the following properties hold:

For the trivial partition $\{a,b\}$ of the interval $[a,b]$, if 
    \[\omega:=\sup_{x\in[a,b]}f(x)-\inf_{x\in[a,b]}f(x),\] then $U_f(\{a,b\})-L_f(\{a,b\})=\omega(b-a)$.

    For every $\Delta\in\mathcal{P}[a,b]$, $L_f(\Delta)\leq U_f(\Delta)$.

    If $\Delta_2\in\mathcal{P}[a,b]$ is a refinement of $\Delta_1\in\mathcal{P}[a,b]$, then $L_f(\Delta_1)\leq L_f(\Delta_2)\leq U_f(\Delta_2)\leq U_f(\Delta_1)$.

    For any $\Delta_1,\Delta_2\in\mathcal{P}[a,b]$, the union $\Delta_1\cup\Delta_2$ belongs to $\mathcal{P}[a,b]$ and is a refinement of both $\Delta_1$ and $\Delta_2$; consequently, $L_f(\Delta_1)\leq L_f(\Delta_1\cup\Delta_2)\leq U_f(\Delta_1\cup\Delta_2)\leq U_f(\Delta_2)$, and in particular $L_f(\Delta_1)\leq U_f(\Delta_2)$.

    If $a<m<b$, $\Delta_1\in\mathcal{P}[a,m]$, and $\Delta_2\in\mathcal{P}[m,b]$, then $\Delta_1\cup\Delta_2\in\mathcal{P}[a,b]$ with $U_f(\Delta_1\cup\Delta_2)=U_f(\Delta_1)+U_f(\Delta_2)$ and $L_f(\Delta_1\cup\Delta_2)=L_f(\Delta_1)+L_f(\Delta_2)$.

    For any $\varepsilon>0$ and any $x\in(a,b]$, if there exists $\Delta_x\in\mathcal{P}[a,x]$ such that $U_f(\Delta_x)-L_f(\Delta_x)<\frac{x-a}{b-a}\varepsilon$, then for every $t\in(a,x]$, there exists $\Delta_t\in\mathcal{P}[a,t]$ such that $U_f(\Delta_t)-L_f(\Delta_t)<\frac{t-a}{b-a}\varepsilon$.

A function $f:[a,b]\into\R$ is \emph{Darboux integrable} over $[a,b]$ if for every $\varepsilon>0$, there exists $\Delta\in\mathcal{P}[a,b]$ such that $U_f(\Delta)-L_f(\Delta)<\varepsilon$. A function $f:[a,b]\into\R$ \emph{has a Darboux integral} over $[a,b]$ if there exists a unique real number $\int_a^b f$ (also written $\int_a^b f(t)\,dt$) such that $L_f(\Delta)\leq\int_a^b f\leq U_f(\Delta)$ for every $\Delta\in\mathcal{P}[a,b]$. For any $c\in[a,b]$, we set $\int_c^c f(t)\,dt:=0$, and for $\alpha,\beta\in[a,b]$ with $\alpha<\beta$, we set $\int_\beta^\alpha f(t)\,dt:=-\int_\alpha^\beta f(t)\,dt$. Directly from the definitions and routine manipulation of inequalities, if $I=\int_a^b f$, then for every $\varepsilon>0$, the numbers $I-\frac{\varepsilon}{2}$ and $I+\frac{\varepsilon}{2}$ are distinct from $I$, so there exist partitions $\Delta_1,\Delta_2\in\mathcal{P}[a,b]$ such that $I-\frac{\varepsilon}{2}<L_f(\Delta_1)$ and\linebreak $U_f(\Delta_2)<I+\frac{\varepsilon}{2}$. Furthermore, when $f$ and $g$ have Darboux integrals over the relevant subintervals of $[a,b]$, the definite integral satisfies the following standard properties:

For any $x,c\in[a,b]$, $\int_a^x f(t)\,dt-\int_a^c f(t)\,dt=\int_c^x f(t)\,dt$, which equals $\int_{\min\{x,c\}}^{\max\{x,c\}} f(t)\,dt$ if $c\leq x$ and $-\int_{\min\{x,c\}}^{\max\{x,c\}} f(t)\,dt$ if $x<c$.

    For any $\alpha,\beta\in[a,b]$ and $k,\lambda,\mu\in\R$, \[\int_\alpha^\beta (\lambda f(t)+\mu g(t))\,dt=\lambda\int_\alpha^\beta f(t)\,dt+\mu\int_\alpha^\beta g(t)\,dt,\] and $\int_\alpha^\beta k\,dt=k(\beta-\alpha)$.

For $\alpha,\beta\in[a,b]$ with $\alpha\leq\beta$, if $f(t)\leq g(t)$ for all $t\in[\alpha,\beta]$, then $\int_\alpha^\beta f(t)\,dt\leq\int_\alpha^\beta g(t)\,dt$; moreover, $|f|$ has a Darboux integral over $[\alpha,\beta]$ and $\left|\int_\alpha^\beta f(t)\,dt\right|\leq\int_\alpha^\beta |f(t)|\,dt$.

\paragraph{Differentiability.} Let $I\sub\R$ be an interval containing more than one point, and let $f:I\into\R$. We say that $f$ is \emph{differentiable at} $c\in I$ if the limit
\[
f'(c):=\lim_{x\rightarrow c}\frac{f(x)-f(c)}{x-c}
\]
exists in $\R$ (that is, for every $\varepsilon>0$, there exists $\delta>0$ such that $x\in I$ and $0<|x-c|<\delta$ imply $\left|\frac{f(x)-f(c)}{x-c}-f'(c)\right|<\varepsilon$), in which case $f'(c)$ is the \emph{derivative} of $f$ at $c$. If $f$ is differentiable at every $c\in I$, then $f$ is \emph{differentiable on} $I$, and the rule of assignment $c\mapsto f'(c)$ is a function called the \emph{derivative} $f'$ of $f$ on $I$. A function $F:[a,b]\into\R$ is an \emph{antiderivative} of $f:[a,b]\into\R$ if $F$ is differentiable on $[a,b]$ and $F'(x)=f(x)$ for all $x\in[a,b]$. Using epsilon-delta arguments and routine manipulation of difference quotients, one obtains the following properties:

If $f$ is differentiable on $I$ and $J\sub I$ is a subinterval containing more than one point, then $f$ is differentiable on $J$ with the same derivative.

    If $f$ is differentiable at $c\in I$ (respectively, differentiable on $I$), then $f$ is continuous at $c$ (respectively, continuous on $I$).

    The derivative of any constant function is $0$; the derivative of the identity function $x\mapsto x$ is $1$; and if $f$ and $g$ are differentiable on $I$ and $\alpha,\beta\in\R$, then $\alpha f+\beta g$ is differentiable on $I$ with $(\alpha f+\beta g)'=\alpha f'+\beta g'$. In particular, $(-f)'=-f'$.

    If $f$ is differentiable on $[a,b]$ with $f'(a)<0$, an epsilon-delta argument may be used to show that there exists $\delta_1>0$ such that $a+\delta_1\in[a,b]$ and $f(a)>f(a+\delta_1)$; similarly, if $f'(b)>0$, there exists $\delta_2>0$ such that $b-\delta_2\in[a,b]$ and $f(b)>f(b-\delta_2)$.

    If $f:[a,b]\into\R$ is differentiable at $c\in(a,b)$ and satisfies either\linebreak $f(c)\leq f(x)$ for all $x\in[a,b]$ or $f(c)\geq f(x)$ for all $x\in[a,b]$, then the one-sided limits of the difference quotient $\frac{f(x)-f(c)}{x-c}$ as $x\rightarrow c^-$ and $x\rightarrow c^+$ may be routinely used to prove that $f'(c)=0$.

Finally, let $I\sub\R$ be an interval. A function $f:I\into\R$ is \emph{strictly increasing} on $I$ if $x,c\in I$ and $x<c$ imply $f(x)<f(c)$; \emph{increasing} on $I$ if $x,c\in I$ and $x<c$ imply $f(x)\leq f(c)$; and \emph{constant} on $I$ if $f(x)=f(c)$ for all $x,c\in I$. By the property of functions that $x=c$ implies $f(x)=f(c)$, we find that $f$ is constant on $I$ if $x,c\in I$ and $x<c$ imply $f(x)=f(c)$. Directly from the definitions and routine manipulation of inequalities, if both $f$ and $-f$ are increasing on $I$, then $f$ is constant on $I$.

\section{Completeness Axioms}\label{AxiomSec}

As a mathematical theory, real function theory must have a starting point. Aside from the definitions, which have been enumerated in the previous section, the axioms or postulates are such starting points, and for real function theory, we have a choice of four, the equivalence of which shall be proven in this section. First, we mention the following.

\begin{enumerate}
    \item[\AP]\emph{Archimedean Property of $\R$.} For each $x\in\R$, there exists $N\in\N$ such that $x<N$.
    \item[\NIP]\emph{Nested Intervals Property of $\R$.} If intervals $[a_n,b_n]$ where $n\in\N$ form a sequence of nested intervals and $\nseqLim (b_n-a_n)=0$, then there exists $c\in\R$ such that $\{c\}=\bigcap_{n=1}^\infty [a_n,b_n]$.
\end{enumerate}

The conjunction $\AP\&\NIP$ is a \emph{completeness axiom} for the ordered field $\R$. But in addition to this, there are three other possible completeness axioms that may be used for naive or elementary real function theory, which we list in the following.

\begin{theorem} The following are equivalent.
\begin{enumerate}
    \item[\AP\&\NIP] [The conjunction of the statements of the \AP\  and the \NIP.]
    \item[\CA]\emph{(Dedekind) Cut Axiom.} Every cut of $\R$ has a cut point.
    \item[\ES]\emph{Existence of Suprema.} If $\theSet$ is a nonempty subset of $\R$, and if there exists $b\in\R$ such that $b$ is an upper bound of $\theSet$, then there exists $c\in\R$ such that $c=\sup\theSet$.
    \item[\BWP]\emph{Bolzano-Weierstrass Property of $\R$.} If $(a_n)$ is a bounded sequence, then $(a_n)$ has a convergent subsequence.
\end{enumerate}
\end{theorem}

The equivalence of the above statements is summarized in the following diagram.

\begin{center}
\begin{tikzpicture}[
scale=0.7,             
    transform shape,       
    double arrow/.style={
        double, 
        double distance=1.5pt, 
        draw=#1, 
        ->, 
        >={Latex[length=4pt, width=5pt, color=#1]}, 
        shorten >= 3pt, 
        shorten <= 3pt,
        rounded corners=10pt 
    },
    vertex/.style={
        draw, 
        rectangle, 
        rounded corners=5pt, 
        minimum width=40pt, 
        minimum height=20pt
    }
]

        \node[vertex, draw=red, fill=red!10] (CA) at (0,0) {\CA};
     \node[vertex, draw=red, fill=red!10] (ES) at (5,0) {\ES};
     \node[vertex, draw=red, fill=red!10] (BWP) at (5,5) {\BWP};
     \node[vertex, draw=red, fill=red!10] (NIP) at (0,5) {\AP\&\NIP};

        \draw[double arrow=red, shorten >= 0pt, shorten <= 0pt] (ES) -- (BWP);
        \draw[double arrow=red, shorten >= 0pt, shorten <= 0pt] (BWP) -- (NIP);
        \draw[double arrow=red, shorten >= 0pt, shorten <= 0pt] (NIP) -- (CA);
        \draw[double arrow=red, shorten >= 0pt, shorten <= 0pt] (CA) -- (ES);
\end{tikzpicture}
\end{center}

\begin{proof}[Proof of $\AP\&\NIP\implies\CA$.] If $(A,B)$ is a cut of $\R$, then there exist $a_1\in A$ and $b_1\in B$ with $a_1<b_1$. Suppose that for some $n\in\N$, intervals $[a_n,b_n]\sub[a_{n-1},b_{n-1}]\sub\cdots\sub[a_1,b_1]$ have been determined such that\linebreak $a_k\in A$ and $b_k\in B$ and $b_k-a_k=2(b_1-a_1)\frac{1}{2^k}$, for all $k\in\{1,2,\ldots,n\}$. If $\frac{a_n+b_n}{2}$ is in $A$, then define $a_{n+1}:=\frac{a_n+b_n}{2}\in A$ and $b_{n+1}:=b_n\in B$. Otherwise, let $a_{n+1}:=a_n\in A$ and $b_{n+1}:=\frac{a_n+b_n}{2}\in B$. In both cases, $[a_{n+1},b_{n+1}]\sub[a_n,b_n]$, with $b_{n+1}-a_{n+1}=\frac{b_{n}-a_n}{2}=2(b_1-a_1)\frac{1}{2^{n+1}}$. By induction, we have produced nested intervals $[a_n,b_n]$ such that $a_n\in A$, $b_n\in B$ and $b_n-a_n=2(b_1-a_1)\frac{1}{2^n}$ for all $n\in\N$.

Given $\varepsilon>0$, \AP\ implies that there exists $N\in\N$ such that $\frac{1}{\varepsilon}<N$. If $n\in\N$ with $n\geq N$, then $\frac{1}{\varepsilon}<N\leq n<2^n$ from which, we obtain $|\frac{1}{2^n}-0|<\varepsilon$. Thus, 
\begin{eqnarray}
    \nseqLim\frac{1}{2^n}&=&0,\label{AP2n}
\end{eqnarray}
so $\nseqLim(b_n-a_n)=2(b_1-a_1)\nseqLim\frac{1}{2^n}=0$. At this point, we may invoke the \NIP\  to deduce that there exists $c\in\R$ such that $\{c\}=\bigcap_{n=1}^\infty [a_n,b_n]$.

If $c$ is not a cut point of $(A,B)$, then either there exists $a\in A$ such that $a>c$ or there exists $b\in B$ such that $b<c$. In the former case, $a-c>0$. Since $\nseqLim(b_n-a_n)=0$, there exists $N\in\N$ such that $b_N-a_N<a-c$. Since $c\in[a_N,b_N]$, we have $b_N-c\leq b_N-a_N<a-c$, which implies $b_N<a$, where $a\in A$ and $b_N\in B$, contradicting the definition of cut. In the second case, $c-b>0$, so there exists $N\in\N$ such that $b_N-a_N<c-b$, and $c\in[a_N,b_N]$ further implies $c-a_N\leq b_N-a_N<c-b$, implying $b<a_N$. Since $a_N\in A$ and $b\in B$, this again contradicts the definition of cut. Therefore, $c$ is a cut point of $(A,B)$.
\end{proof}

\begin{proof}[Proof of $\CA\implies\ES$.] Suppose $S$ is a nonempty subset of $\R$ such that the set $B$ of all the upper bounds of $S$ is nonempty. If $A:=\R\setdiff B$, then $A\cap B=\emptyset$ and $A\cup B=\R$. Since $S$ is nonempty, there exists $s\in S$. Since $s-1<s$, the number $s-1$ is not an upper bound of $S$, so $s-1\in A$, and $A$ is nonempty. If $a\in A$ and $b\in B$, then $a$ is not an upper bound of $S$, so there exists $t\in S$ such that $a<t$. Since $b\in B$ implies $b$ is an upper bound of $S$, we further have $a<t\leq b$, so $a<b$. At this point, we have proven that $(A,B)$ is a cut of $\R$, and by \CA, has a cut point $c$. If $c$ is not an upper bound of $S$, then there exists $u\in S$ such that $c<u$. If $u\in A$, then by the definition of cut point, $u\leq c<u$, leading to the contradiction $u<u$. Thus, $u\in B$. That is, $u$ is an element and an upper bound of $S$. This implies $u=\sup S$, and we are done. If $c$ is an upper bound of $S$, then by the definition of cut point and the definition of $B$, we have $c\leq b$ for any upper bound $b$ of $S$. This implies $c=\sup S$, and this completes the proof.
\end{proof}

\begin{proof}[Proof of $\ES\implies\BWP$.] We first prove that \ES\ implies \AP. If \AP\  is false, then there exists $\xi\in\R$ such that $N\in\N$ implies $\xi\geq N$, which means that $\xi$ is an upper bound of $\N$. By \ES, there exists $u\in\R$ such that $u=\sup\N$. Since $u-\frac{1}{2}<u$, the number $u-\frac{1}{2}$ is not an upper bound of $\N$, so there exists $M\in\N$ such that $M>u-\frac{1}{2}$. Adding $1$ to both sides, and comparing the resulting right-hand side with $u$, we obtain $M+1>u+\frac{1}{2}>u$, where $M\in\N$ implies $M+1\in\N$. This implies $u$ is not an upper bound of $\N$, contradicting $u=\sup\N$. Therefore, \ES\ implies \AP.

Given a bounded sequence $(a_n)$, let \[S:= \left\{ x \in \mathbb{R} : \left| \{ n \in \mathbb{N} : a_n > x \} \right| \not< \infty \right\}.\]
Because $(a_n)$ is bounded, there exists an interval $[a,b]$ such that $a\leq a_n\leq b$ for all $n\in\N$. Thus, $a_n>a-1$ for all $n\in\N$, which implies $\N\sub\{ n \in \mathbb{N} : a_n > a-1 \}$, so the set at the right-hand side cannot have finite cardinality. This means that $a-1\in S$, so $S$ is nonempty. If $x\in S$, since $(a_n)$ is a sequence in $[a,b]$, we have $x<a_n\leq b$ for all $n\in\N$. Hence, $b$ is an upper bound of $S$. By \ES, there exists $c\in\R$ such that $c=\sup S$.

For each $k\in\N$, since $c<c+\frac{1}{k}$, the number $c+\frac{1}{k}$ is not an element of $S$. By the definition of $S$, the set $\{n\in\N\,:\,a_n>c+\frac{1}{k}\}$ is finite. On the other hand, since $c-\frac{1}{k}<c$, the number $c-\frac{1}{k}$ is not an upper bound of $S$, so there exists $x_k\in S$ such that $x_k>c-\frac{1}{k}$. Since $x_k\in S$, the set $\{n\in\N\,:\,a_n>x_k\}$ has infinitely many elements, and this set is a subset of $\{n\in\N\,:\,a_n>c-\frac{1}{k}\}$, because $a_n>x_k>c-\frac{1}{k}$. Thus, the set $\{n\in\N\,:\,a_n>c-\frac{1}{k}\}$ also has infinitely many elements. Consequently, the set
\begin{eqnarray}
    T_k &:=& \left\{ n\in\N \,:\, c-\frac{1}{k} < a_n \leq c+\frac{1}{k} \right\},\nonumber\\
    &=& \left\{n\in\N\,:\,a_n>c-\frac{1}{k}\right\}\setdiff\left\{n\in\N\,:\,a_n>c+\frac{1}{k}\right\},\nonumber
\end{eqnarray}
has infinitely many elements for every $k\in\N$.

Since $T_1$ is an infinite and hence nonempty subset of $\N$, it has a least element $N_1$. By the definition of $T_1$, we have $c-1<a_{N_1}\leq c+1$. Suppose that for some $k\in\N$, the indices $N_1 < N_2 < \dots < N_k$ have been chosen such that $N_h\in T_h$, and hence $c-\frac{1}{h}<a_{N_h}\leq c+\frac{1}{h}$, for all $h\in\{1,2,\ldots,k\}$. Since $T_{k+1}$ has infinitely many elements and since $\{1,2,\ldots,N_k\}$ is finite, the set $T_{k+1}\setdiff\{1,2,\ldots,N_k\}=\{n\in T_{k+1} \,:\, n > N_k\}$ is an infinite, and hence nonempty, subset of $\N$. We define $N_{k+1}$ to be the least element of this set. Thus, $N_1 < N_2 < \dots < N_k<N_{k+1}$, and $N_{k+1}\in T_{k+1}$ implies $c-\frac{1}{k+1}<a_{N_{k+1}}\leq c+\frac{1}{k+1}$. By induction, we have produced a subsequence $\left(a_{N_k}\right)$ of $(a_n)$ such that $c-\frac{1}{k}<a_{N_{k}}\leq c+\frac{1}{k}$, which implies $|a_{N_k}-c|<\frac{1}{k}$ for all $k\in\N$. Given $\varepsilon>0$, by \AP, there exists $K\in\N$ such that $\frac{1}{\varepsilon}<K$. If $k\in\N$ with $k\geq K$, then $|a_{N_k}-c|<\frac{1}{k}\leq\frac{1}{K}<\varepsilon$. Therefore, the subsequence $\left(a_{N_k}\right)$ is convergent.
\end{proof}

\begin{proof}[Proof of $\BWP\implies\AP$.] If \AP\  is false, then there exists $\xi\in\R$ such that $N\in\N$ implies $\xi\geq N$. Since any such $N$ is at least $1$, the sequence $(n)$ is in $[1,\xi]$. By the \BWP, the sequence $(n)$ has a convergent subsequence $(N_k)$. Let $L:=\kseqLim N_k$. Since $\frac{1}{4}>0$, there exists $K\in\N$ such that for $k\in\{K,K+1\}$, we have $|N_k-L|<\frac{1}{4}$. Since $(n)$ is strictly increasing, so is $(N_k)$, with \[1\leq N_{K+1}-N_K=|N_{K+1}-N_K|\leq |N_{K+1}-L|+|L-N_K|<\frac{1}{4}+\frac{1}{4}=\frac{1}{2},\] leading to the contradiction $1<\frac{1}{2}$.
\end{proof}

\begin{proof}[Proof of $\BWP\implies\NIP$.] If the intervals $[a_n,b_n]$ where $n\in\N $ form a sequence of nested intervals, then $a_1\leq a_n\leq b_m\leq b_1$ for all $m,n\in\N$, so the sequence $(a_n)$ is bounded, and by the \BWP, has a convergent subsequence $\left(a_{N_k}\right)$. From $a_n\leq b_m$ for all $m,n\in\N$, we obtain $a_{N_k}\leq b_m$ for all $m,k\in\N$, which implies $\kseqLim a_{N_k}\leq b_m$ for all $m\in\N$. If there exists $m\in\N$ such that $a_m>\kseqLim a_{N_k}$, then there exists $K\in\N$ such that for all $h\geq K$, we have $|a_{N_h}-\kseqLim a_{N_k}|<a_m-\kseqLim a_{N_k}$, which implies $a_{N_h}<a_m$ for all $h\geq K$. From the definition of subsequence, $m\leq N_m$, and since $(a_n)$ is monotonically increasing,\linebreak $a_{N_h}<a_m\leq a_{N_m}$ for all $h\geq K$. In particular, $a_{N_h}< a_{N_m}$ for all $h\geq \max\{K,m\}\geq m$. From the definition of subsequence, $m\leq h$ implies $N_m\leq N_h$, and since $(a_n)$ is monotonically increasing, $a_{N_m}\leq a_{N_h}$, contradicting $a_{N_h}<a_{N_m}$ from earlier. Hence, $a_m\leq \kseqLim a_{N_k}$ for all $m\in\N$. At this point, we have proven that $\kseqLim a_{N_k}\in\bigcap_{m=1}^\infty[a_m,b_m]$, so this intersection is nonempty.

If $\bigcap_{n=1}^\infty[a_n,b_n]$ has two distinct elements $c_1$ and $c_2$ that further satisfy, without loss of generality, $c_1<c_2$, then for any $N\in\N$, we may take $n=N$ for which $a_n\leq c_1<c_2\leq b_n$, so $|(b_n-a_n)-0|=b_n-a_n\geq c_2-c_1>0$, and this proves that $(b_n-a_n)$ does not converge to $0$. By contraposition, $\nseqLim (b_n-a_n)=0$ implies that $\bigcap_{n=1}^\infty[a_n,b_n]$ has exactly one element.
\end{proof}

\section{Compactness Arguments using Nested Intervals}\label{NestedSec}
Our goal in this section is to exhibit proofs of the following real analysis theorems. The method that shall be used is by nested intervals (and $\AP$), and this corresponds to the first completeness axiom, $\AP\&\NIP$, from Section~\ref{AxiomSec}. The pattern shall persist in the next three sections where we will use methods corresponding to the other completeness axioms. We now give the statements we shall prove.

\begin{enumerate}
    \item[\IVT]\emph{Intermediate Value Theorem.} If $f$ is continuous on $[a,b]$ and\linebreak $f(a)<0<f(b)$, then there exists $c\in[a,b]$ such that $f(c)=0$.
    \item[\DIT]\emph{Darboux Integrability Theorem.} If $f$ is continuous on $[a,b]$, then $f$ is \emph{Darboux integrable} of $[a,b]$: for each $\varepsilon>0$, there exists a partition $\Delta$ of $[a,b]$ such that the difference of the upper and lower Darboux sums of $f$ over $[a,b]$ satisfy the inequality $U_f(\Delta)-L_f(\Delta)<\varepsilon$.
    \item[\IET]\emph{Darboux Integral Theorem.} If $f$ is continuous on $[a,b]$, then $f$ \emph{has a Darboux integral} over $[a,b]$: there exists a unique $\int_a^bf\in\R$ such that for any partition $\Delta$ of $[a,b]$, we have $L_f(\Delta)\leq \int_a^bf\leq U_f(\Delta)$.
\end{enumerate}

\begin{proof}[Proof of the \IVT\ using nested intervals.]
Let $a_1:=a$ and $b_1:=b$, so by assumption $f(a_1)<0<f(b_1)$. If $f\left(\frac{a_1+b_1}{2}\right)=0$, then $c:=\frac{a_1+b_1}{2}\in[a,b]$ satisfies $f(c)=0$, and we are done. Otherwise, suppose that for some $n\in\N$, nested intervals $[a_n,b_n]\sub[a_{n-1},b_{n-1}]\sub\cdots\sub[a_1,b_1]$ have been determined such that $f(a_k)<0<f(b_k)$ and $b_k-a_k=2(b-a)\frac{1}{2^{k}}$ for all $k\in\{1,2,\ldots,n\}$. If $f\left(\frac{a_n+b_n}{2}\right)=0$, then $c:=\frac{a_n+b_n}{2}\in[a,b]$ satisfies $f(c)=0$, and we are done. Otherwise, we proceed with the induction. If $f\left(\frac{a_n+b_n}{2}\right)<0$, we set $a_{n+1}:=\frac{a_n+b_n}{2}$ and $b_{n+1}:=b_n$. If $f\left(\frac{a_n+b_n}{2}\right)>0$, we set $a_{n+1}:=a_n$ and $b_{n+1}:=\frac{a_n+b_n}{2}$. In both cases, the interval $[a_{n+1},b_{n+1}]\sub[a_n,b_n]$ satisfies the conditions $f(a_{n+1})<0<f(b_{n+1})$ and $b_{n+1}-a_{n+1}=\frac{b_n-a_n}{2}=2(b-a)\frac{1}{2^{n+1}}$. By induction, we have produced a sequence of nested intervals $[a_n,b_n]$ such that $f(a_n)<0<f(b_n)$ and $b_n-a_n=2(b-a)\frac{1}{2^{n}}$ for all $n\in\N$. By \AP, $\nseqLim(b_n-a_n)=2(b-a)\nseqLim\frac{1}{2^{n}}=0$. By the \NIP, there exists $c\in[a,b]$ such that $\{c\}=\bigcap_{n=1}^\infty [a_n,b_n]$. 

Given $n\in\N$, the condition $a_n\le c\le b_n$ implies $0\leq c-a_n\leq b_n-a_n$. Taking limits as $n\into\infty$, and using $\nseqLim(b_n-a_n)=0$ from earlier, by the Squeeze Theorem for sequence limits, $\nseqLim(c-a_n)=0$. Using the definition of sequence convergence, this implies $\nseqLim a_n=c$. Since the sequences\linebreak $(b_n-a_n)$ and $(a_n)$ are now both known to be convergent, so is their sum, which is the sequence $(b_n)$, with $\nseqLim b_n=0+c=c$. That is, we have at this point that $\nseqLim b_n = c = \nseqLim a_n$. Since $f$ is continuous on $[a,b]$, $f$ is continuous at $c$, so $\nseqLim f(b_n) = f(c)=\nseqLim f(a_n)$. Since $0<f(b_n)$ and $f(a_n)<0$ for all $n\in\N$, we further have $0\leq \nseqLim f(b_n) = f(c)=\nseqLim f(a_n) \leq 0$. Therefore, $f(c)=0$, as desired.
\end{proof}

\begin{proof}[Proof of the \DIT\ using nested intervals.]
Fix $\varepsilon>0$. Suppose for contradiction that there exists no partition $\Delta$ of $[a,b]$ such that $U_f(\Delta)-L_f(\Delta)<\varepsilon$. Let $[a_1,b_1]:=[a,b]$. Suppose that for some $n\in\N$, nested intervals \[[a_n,b_n]\sub[a_{n-1},b_{n-1}]\sub\cdots\sub[a_1,b_1]\] have been determined such that $b_k-a_k=2(b-a)\frac{1}{2^{k}}$ and no partition $\Delta_k$ of $[a_k,b_k]$ satisfies $U_f(\Delta_k)-L_f(\Delta_k)<\frac{\varepsilon}{2^{k-1}}$ for all $k\in\{1,2,\ldots,n\}$. If there existed partitions $\Delta_1$ of $\left[a_n,\frac{a_n+b_n}{2}\right]$ and $\Delta_2$ of $\left[\frac{a_n+b_n}{2},b_n\right]$ such that $U_f(\Delta_1)-L_f(\Delta_1)<\frac{\varepsilon}{2^n}$ and $U_f(\Delta_2)-L_f(\Delta_2)<\frac{\varepsilon}{2^n}$, then $\Delta:=\Delta_1\cup\Delta_2$ would be a partition of $[a_n,b_n]$ satisfying
\[ U_f(\Delta)-L_f(\Delta) = \big(U_f(\Delta_1)-L_f(\Delta_1)\big) + \big(U_f(\Delta_2)-L_f(\Delta_2)\big) < \frac{\varepsilon}{2^n} + \frac{\varepsilon}{2^n} = \frac{\varepsilon}{2^{n-1}}, \]
contradicting our hypothesis on $[a_n,b_n]$. Thus, we define $[a_{n+1},b_{n+1}]$ to be whichever of $\left[a_n,\frac{a_n+b_n}{2}\right]$ or $\left[\frac{a_n+b_n}{2},b_n\right]$ has no partition $\Lambda$ that satisfies $U_f(\Lambda)-L_f(\Lambda)<\frac{\varepsilon}{2^n}$. In both cases, $[a_{n+1},b_{n+1}]\sub[a_n,b_n]$, $b_{n+1}-a_{n+1}=\frac{b_n-a_n}{2}=2(b-a)\frac{1}{2^{n+1}}$, and no partition $\Delta$ of $[a_{n+1},b_{n+1}]$ satisfies the inequality $U_f(\Delta)-L_f(\Delta)<\frac{\varepsilon}{2^n}$.

By induction, we obtain a sequence of nested intervals $[a_n,b_n]$ such that $b_n-a_n=2(b-a)\frac{1}{2^{n}}$ and no partition $\Delta_n$ of $[a_n,b_n]$ satisfies the inequality $U_f(\Delta_n)-L_f(\Delta_n)<\frac{\varepsilon}{2^{n-1}}$ for all $n\in\N$. By \AP, $\nseqLim(b_n-a_n)=2(b-a)\nseqLim\frac{1}{2^{n}}=0$. By the \NIP, there exists $c\in[a,b]$ such that $\{c\}=\bigcap_{n=1}^\infty [a_n,b_n]$.

Since $f$ is continuous on $[a,b]$, $f$ is continuous at $c$, and since\linebreak  $\frac{\varepsilon}{4(b-a)}>0$, there exists $\delta>0$ such that if $x\in[a,b]$ satisfies $|x-c|<\delta$, then $|f(x)-f(c)|<\frac{\varepsilon}{4(b-a)}$. By \AP, there exists $N\in\N$ such that $2\frac{b-a}{\delta}<N<2^{N}$, so $b_N-a_N=2(b-a)\frac{1}{2^{N}}<\delta$. Since $c\in[a_N,b_N]$, for any $x\in[a_N,b_N]$, we have $|x-c|\leq b_N-a_N<\delta$, so $|f(x)-f(c)|<\frac{\varepsilon}{4(b-a)}$. This further implies $f(c)-\frac{\varepsilon}{4(b-a)}\leq f(x)\leq f(c)+\frac{\varepsilon}{4(b-a)}$, where $f(x)$ is an arbitrary element of the image set $\{f(x)\,:\,x\in[a_N,b_N]\}$. Thus, $f(c)-\frac{\varepsilon}{4(b-a)}$ is a lower bound, and $f(c)+\frac{\varepsilon}{4(b-a)}$ is an upper bound of said image set. Consequently, $f(c)-\frac{\varepsilon}{4(b-a)}\leq \inf_{x\in[a_N,b_N]}f(x)$ and $\sup_{x\in[a_N,b_N]}f(x)\le f(c)+\frac{\varepsilon}{4(b-a)}$, from which we obtain
\[\sup_{x\in[a_N,b_N]}f(x)-\inf_{x\in[a_N,b_N]}f(x)\leq  f(c)+\frac{\varepsilon}{4(b-a)}-f(c)+\frac{\varepsilon}{4(b-a)}=\frac{\varepsilon}{2(b-a)}.\]

Now consider the trivial partition $\Delta_N:=\{a_N,b_N\}$ of $[a_N,b_N]$. The difference of upper and lower Darboux sums for $\Delta_N$ satisfies
\begin{eqnarray}
  U_f(\Delta_N)-L_f(\Delta_N) &=&   \left( \sup_{x\in[a_N,b_N]} f(x) - \inf_{x\in[a_N,b_N]} f(x) \right) (b_N-a_N),\nonumber\\
  &\le& \frac{\varepsilon}{2(b-a)} \cdot 2(b-a)\frac{1}{2^{N}}=\frac{1}{2}\cdot\frac{\varepsilon}{2^{N-1}}<\frac{\varepsilon}{2^{N-1}},\nonumber
\end{eqnarray}
contradicting the condition that no partition of $[a_N,b_N]$ satisfies the inequality $U_f(\Delta_N)-L_f(\Delta_N)<\frac{\varepsilon}{2^{N-1}}$. Therefore, there exists a partition $\Delta$ of $[a,b]$ satisfying $U_f(\Delta)-L_f(\Delta)<\varepsilon$.
\end{proof}

\begin{proof}[Proof of $\DIT\implies\IET$ using nested intervals.] If $f$ is continuous on $[a,b]$, then by \DIT, there exists a partition $\Delta_1$ of $[a,b]$ such that\linebreak $U_f(\Delta_1)-L_f(\Delta_1)<1$. If one of $U_f(\Delta_1)=\infty$ or $L_f(\Delta_1)=-\infty$ is true, then we obtain the contradiction $\infty<1$. Thus, $-\infty<L_f(\Delta_1)\leq U_f(\Delta_1)<\infty$, so $L_f(\Delta_1),U_f(\Delta_1)\in\R$. Suppose that for some $n\in\N$, partitions\linebreak $\Delta_1\sub\Delta_2\sub\cdots\sub\Delta_n$ of $[a,b]$ have been identified such that\linebreak $U_f(\Delta_k)-L_f(\Delta_k)<\frac{1}{k}$ for all $k\in\{1,2,\ldots,n\}$. Using \DIT\  again, there exists a partition $\Omega$ of $[a,b]$ such that $U_f(\Omega)-L_f(\Omega)<\frac{1}{n+1}$. We define $\Delta_{n+1}:=\Delta_n\cup\Omega$, which is also a partition of $[a,b]$, with $\Delta_n\sub\Delta_{n+1}$, and also, $\Omega\sub\Delta_{n+1}$, or that $\Delta_{n+1}$ is a refinement of $\Omega$, so $L_f(\Omega)\leq L_f(\Delta_{n+1})\leq U_f(\Delta_{n+1})\leq U_f(\Omega)$. By routine manipulation of inequalities, $U_f(\Delta_{n+1})-L_f(\Delta_{n+1})\leq U_f(\Omega)-L_f(\Omega)<\frac{1}{n+1}$. By induction, we have produced a sequence $(\Delta_n)$ of partitions of $[a,b]$ such that $0\leq U_f(\Delta_n)-L_f(\Delta_n)<\frac{1}{n}$ and $\Delta_n\sub\Delta_{n+1}$ for all $n\in\N$. The latter condition implies $L_f(\Delta_n)\leq L_f(\Delta_{n+1})\leq U_f(\Delta_{n+1})\leq U_f(\Delta_n)$. Thus, the sequence $([L_f(\Delta_n),U_f(\Delta_n)])$ of closed and bounded intervals is nested. By \AP, $0\leq \nseqLim(U_f(\Delta_n)-L_f(\Delta_n))\leq \nseqLim\frac{1}{n}=0$, so $\nseqLim(U_f(\Delta_n)-L_f(\Delta_n))=0$. By \NIP, there exists $\int_a^bf\in\R$ such that $\bigcap_{n=1}^\infty[L_f(\Delta_n),U_f(\Delta_n)]=\{\int_a^bf\}$. Thus, for any $n\in\N$, we have $L_f(\Delta_n)\leq \int_a^bf\leq U_f(\Delta_n)$, which implies $0\leq\int_a^bf-L_f(\Delta_n)\leq U_f(\Delta_n)-L_f(\Delta_n)<\frac{1}{n}$, where the right-most member, by routine use of \AP\  and the definition of sequence convergence, converges to $0$. By the Squeeze Theorem for, and the linearity property of, sequence limits, $\nseqLim L_f(\Delta_n)=\int_a^bf$. Since $U_f(\Delta_n)=(U_f(\Delta_n)-L_f(\Delta_n))+L_f(\Delta_n)$, we have $\nseqLim U_f(\Delta_n)=0+\int_a^bf=\int_a^bf$. Given an arbitrary partition $\Delta$ of $[a,b]$, for each $n\in\N$, the partition $\Delta\cup\Delta_n$ is a refinement of both $\Delta$ and $\Delta_n$. That is, $\Delta\sub\Delta\cup\Delta_n$ and $\Delta_n\sub\Delta\cup\Delta_n$, so
\begin{eqnarray}
L_f(\Delta) \leq L_f(\Delta\cup\Delta_n) &\leq& U_f(\Delta\cup\Delta_n) \leq U_f(\Delta_n),\nonumber\\
L_f(\Delta_n) \leq L_f(\Delta\cup\Delta_n) &\leq& U_f(\Delta\cup\Delta_n) \leq U_f(\Delta),\nonumber
\end{eqnarray}
which imply $L_f(\Delta)\leq U_f(\Delta_n)$ and $L_f(\Delta_n)\leq U_f(\Delta)$, respectively. By the monotonicity of sequence limits, we further obtain
\[
L_f(\Delta) \leq \nseqLim U_f(\Delta_n) = \int_a^b f = \nseqLim L_f(\Delta_n) \leq U_f(\Delta).
\]
If $I\in\R$ also satisfies $L_f(\Delta)\leq I\leq U_f(\Delta)$ for every partition $\Delta$ of $[a,b]$, then, in particular, taking $\Delta=\Delta_n$ for each $n\in\N$ gives $L_f(\Delta_n)\leq I\leq U_f(\Delta_n)$, which means $I\in\bigcap_{n=1}^\infty[L_f(\Delta_n),U_f(\Delta_n)]=\left\{\int_a^bf\right\}$, and so, $I=\int_a^b f$.
\end{proof}

\section{Compactness Arguments using Dedekind Cuts}\label{CutSec}

We now turn to our second choice of completeness axiom, \CA, and exhibit how the order-theoretic structure of Dedekind cuts may be used to establish the same three foundational results. Rather than bisecting intervals sequentially, the compactness arguments in this section proceed by partitioning $\R$ into a cut tailored to the property under consideration and using its cut point at a crucial step in the proof. For ease of reference, we restate the statements to be proven.

\begin{enumerate}
    \item[\IVT]\emph{Intermediate Value Theorem.} If $f$ is continuous on $[a,b]$ and\linebreak $f(a)<0<f(b)$, then there exists $c\in[a,b]$ such that $f(c)=0$.
    \item[\DIT]\emph{Darboux Integrability Theorem.} If $f$ is continuous on $[a,b]$, then $f$ is \emph{Darboux integrable} over $[a,b]$: for each $\varepsilon>0$, there exists a partition $\Delta$ of $[a,b]$ such that the difference of the upper and lower Darboux sums of $f$ over $[a,b]$ satisfy the inequality $U_f(\Delta)-L_f(\Delta)<\varepsilon$.
    \item[\IET]\emph{Darboux Integral Theorem.} If $f$ is continuous on $[a,b]$, then $f$ \emph{has a Darboux integral} over $[a,b]$: there exists a unique $\int_a^bf\in\R$ such that for any partition $\Delta$ of $[a,b]$, we have $L_f(\Delta)\leq \int_a^bf\leq U_f(\Delta)$.
\end{enumerate}

\begin{proof}[Proof of the \IVT\ using Dedekind cuts.] The desired conclusion is equivalent to the assertion that if $f$ is never zero on $[a,b]$ and negative at $a$, then $f$ is not positive at $b$. Let
\[S := \{x\in[a,b]\,:\,t\in[a,x]\implies f(t)<0\}.\]

If we define $A:=(-\infty,a)\cup S$ and $B:=\R\setdiff A$, then $A\cap B=\emptyset$, $A\cup B=\R$, $a-1\in A\neq\emptyset$, and, because $A=(-\infty,a)\cup S\sub (-\infty,a)\cup [a,b]$ implies $(b,\infty)\sub A^c=B$, we have $b+1\in B\neq\emptyset$. By the definition of $S$, the condition $x\in S$ implies $[a,x]\sub S$, so $x\in A$ implies $(-\infty,x]\sub A$. Thus, if there exist $u\in A$ and $v\in B$ such that $u\geq v$, then $v\in (-\infty,u]\sub A$, which implies $v\in A\cap B$, contradicting $A\cap B=\emptyset$. At this point, we have proven that $(A,B)$ is a cut of $\R$. By \CA, there exists $c\in\R$ that is a cut point of $(A,B)$.

Using the continuity of $f$ at $a$ and the condition $f(a)<0$, by an epsilon-delta argument, there exists $\delta_1>0$ such that $a+\delta_1\in[a,b]$ and $t\in[a,a+\delta_1]$ implies $f(t)<0$. Thus, $a+\delta_1\in S\sub A$. Since $c$ is a cut point of $(A,B)$, we have $a<a+\delta_1\leq c$. As proven earlier, $(b,\infty)\sub B$. Thus, given $\varepsilon>0$, the number $b+\frac{\varepsilon}{2}$ is in $B$, and since $c$ is the cut point of $(A,B)$, we have $c\leq b+\frac{\varepsilon}{2}<b+\varepsilon$. Thus, $c<b+\varepsilon$ for any $\varepsilon>0$, and this implies $c\leq b$. Hence, $a<a+\delta_1\leq c\leq b$.

By the continuity of $f$ at $c$, there exists $\delta_2>0$ such that $a<a+\delta_2<c$ and that $f$ has constant sign on $[a+\delta_2,c]$. Since $c$ is the cut point of $(A,B)$, we have $(-\infty,c)\sub A$, and since $a<a+\delta_2<c$, we find that $a+\delta_2\in A$ but $a+\delta_2\notin(-\infty,a)$. Using the definition of $A$, this implies $a+\delta_2\in S$, which further implies $f$ is negative at $a+\delta_2$. But since $f$ has constant sign on $[a+\delta_2,c]$, then $f$ is negative on $[a+\delta_2,c]$. Consequently, $f$ is negative on $[a,a+\delta_2]\cup[a+\delta_2,c]=[a,c]$. At this point, we have proven $c\in S$. 

Tending towards a contradiction, suppose $c<b$. Since $c\in S$, we have $f(c)<0$, and the continuity of $f$ at $c$ may be used in an epsilon-delta argument to prove that there exists $\delta_3>0$ such that $a<c<c+\delta_3\leq b$ and that $f$ is negative on $[c,c+\delta_3]$. The condition $c\in S$ implies $f$ is negative on $[a,c]$. Hence, $f$ is negative on $[a,c]\cup[c,c+\delta_3]=[a,c+\delta_3]$. At this point, we have produced an element $c+\delta_3$ of $S\sub A$ that is bigger than cut point $c$, a contradiction. Hence, $c{\not<}b$. Since $a\leq c\leq b$, we further have $c=b$. At this point, we have proven that $b\in S$, so $f(b)<0$, as desired.
\end{proof}

\begin{proof}[Proof of the \DIT\ using Dedekind cuts.] Let $\mathcal{P}[a,x]$ denote the set of all partitions of $[a,x]$, and define
\[S := \left\{ x\in[a,b] \,:\, \exists \Delta_x \in \mathcal{P}[a,x]\quad U_f(\Delta_x)-L_f(\Delta_x) < \frac{x-a}{b-a}\varepsilon\right\}.\]
If we define $A:=(-\infty,a)\cup S$ and $B:=\R\setdiff A$, then $A\cap B=\emptyset$ and\linebreak $A\cup B=\R$. Because $a-1\in A$, we find that $A$ is nonempty, and because $A=(-\infty,a)\cup S\sub (-\infty,a)\cup [a,b]$ implies $(b,\infty)\sub A^c=B$, we have $b+1\in B\neq\emptyset$. By the definition of $S$ and the definition of upper and lower Darboux sums, the condition $x\in S$ implies $[a,x]\sub S$, so $x\in A$ implies $(-\infty,x]\sub A$. Thus, if there exist $u\in A$ and $v\in B$ such that $u\geq v$, then $v\in (-\infty,u]\sub A$, which implies $v\in A\cap B$, contradicting $A\cap B=\emptyset$. At this point, we have proven that $(A,B)$ is a cut of $\R$. By \CA, there exists $c\in\R$ that is a cut point of $(A,B)$.

Using the continuity of $f$ at $a$, for $\frac{\varepsilon}{4(b-a)}>0$, there exists $\delta_1>0$ such that $a<a+\delta_1\leq b$ and that $t\in[a,a+\delta_1]$ implies $|f(t)-f(a)|<\frac{\varepsilon}{4(b-a)}$. If we let $t_0:=a$ and $t_1:=\min\{a+\delta_1,b\}$, then the trivial partition $\Delta_1:=\{t_0,t_1\}$ of $[a,a+\delta_1]$ satisfies
\begin{eqnarray}
    U_f(\Delta_1)-L_f(\Delta_1) &=& \left(\sup_{t\in[a,t_1]}f(t)-\inf_{t\in[a,t_1]}f(t)\right)(t_1-a),\nonumber\\
    &\le& \frac{\varepsilon}{2(b-a)}(t_1-a) < \frac{t_1-a}{b-a}\varepsilon.\nonumber
\end{eqnarray}
Thus, $t_1 \in S\sub A$. As proven earlier, $(b,\infty)\sub B$, and using the definition of cut point and an epsilon argument, this implies $c\leq b$. Hence, $a<a+\delta_1\leq c\leq b$.

By the continuity of $f$ at $c$, for $\frac{\varepsilon}{4(b-a)}>0$, there exists $\delta_2>0$ such that $a < a+\delta_2 < c$ and $|f(t)-f(c)| < \frac{\varepsilon}{4(b-a)}$ for all $t \in [a+\delta_2,c]$. This implies that if $M_1:=\sup_{t\in [a+\delta_2,c]}f(t)$ and $m_1:=\inf_{t\in [a+\delta_2,c]}f(t) $, then $M_1-m_1<\frac{\varepsilon}{b-a}$. Since $a+\delta_2\in S$, there exists a partition $\Delta_2$ of $[a,a+\delta_2]$ such that $U_f(\Delta_2)-L_f(\Delta_2)<\frac{\delta_2}{b-a}\varepsilon$. Hence, the partition $\Delta_2\cup\{c\}$ of $[a,c]$ satisfies
\begin{eqnarray}
U_f(\Delta_2\cup\{c\})-L_f(\Delta_2\cup\{c\}) &=& U_f(\Delta_2)-L_f(\Delta_2)\nonumber\\
&&+ (M_1-m_1)(c-a-\delta_2),\nonumber \\
&<& \frac{\delta_2}{b-a}\varepsilon + \frac{\varepsilon}{b-a}(c-a-\delta_2)=\frac{c-a}{b-a}\varepsilon.\nonumber
\end{eqnarray}
At this point, we have proven that $c\in S$.

Tending towards a contradiction, suppose $c<b$. By the continuity of $f$ at $c$, there exists $\delta_3>0$ such that $a<c<c+\delta_3\le b$ and if $M_2:=\sup_{t\in [c, c+\delta_3]}f(t)$ and $m_2:=\inf_{t\in [c, c+\delta_3]}f(t)$, then $M_2-m_2<\frac{\varepsilon}{b-a}$. Considering the partition $\Delta_3:= \Delta_2\cup\{c,c+\delta_3\}$ of $[a, c+\delta_3]$, we obtain
\begin{eqnarray}
U_f(\Delta_3)-L_f(\Delta_3) &=& U_f(\Delta_2\cup\{c\})-L_f(\Delta_2\cup\{c\})+ (M_2-m_2)\delta_3,\nonumber \\
&<& \frac{c-a}{b-a}\varepsilon + \frac{\varepsilon}{b-a}\delta_3=\frac{(c+\delta_3)-a}{b-a}\varepsilon.\nonumber
\end{eqnarray}
At this point, we have produced an element $c+\delta_3$ of $S\sub A$ that is strictly greater than the cut point $c$ of $(A,B)$, a contradiction. Hence, $c \not< b$. Since $a \le c \le b$, we further have $c = b$. Consequently, $b \in S$, which implies the desired conclusion.
\end{proof}

\begin{proof}[Proof of $\DIT\implies\IET$ using Dedekind cuts.] Let $\mathcal{P}[a,b]$ denote the collection of all partitions of $[a,b]$. If $f$ is continuous on $[a,b]$, then by \DIT, there exists $\Omega\in\mathcal{P}[a,b]$ such that $U_f(\Omega)-L_f(\Omega)<1$, which, if one of $U_f(\Omega)=\infty$ or $L_f(\Omega)=-\infty$ is true, leads to the contradiction $\infty<1$. Thus, $L_f(\Omega)\leq U_f(\Omega)<\infty$ and $-\infty<L_f(\Omega)$, and we now find that $L_f(\Omega),U_f(\Omega)\in\R$, so the set $S=\{L_f(\Delta)\in\R\,:\,\Delta\in\mathcal{P}[a,b]\}$ is a nonempty subset of $\R$.

Let $B$ be the set of all upper bounds of $S$ in $\R$, and define $A:=\R\setdiff B$, so that $A\cap B=\emptyset$ and $A\cup B=\R$. Since $L_f(\Omega)\in S$ and $L_f(\Omega)-1<L_f(\Omega)$, the number $L_f(\Omega)-1$ is not an upper bound of $S$, so $L_f(\Omega)-1\in A$, and $A$ is nonempty. Since $L_f(\Lambda)\leq U_f(\Omega)$ for any $\Lambda\in\mathcal{P}[a,b]$, we find that $U_f(\Omega)\in B$, so $B$ is nonempty. If $u\in A$ and $v\in B$, then $u$ is not an upper bound of $S$, so there exists $\Lambda\in\mathcal{P}[a,b]$ with $L_f(\Lambda)\in S$ such that $u<L_f(\Lambda)$. Since $v\in B$ implies $v$ is an upper bound of $S$, we further have $u<L_f(\Lambda)\leq v$, so $u<v$. At this point, we have proven that $(A,B)$ is a cut of $\R$. By \CA, there exists $\int_a^bf\in\R$ that is a cut point of $(A,B)$.

Given $\Delta\in\mathcal{P}[a,b]$, if $L_f(\Delta)=-\infty$, then $L_f(\Delta)\leq\int_a^bf$ holds immediately; otherwise $L_f(\Delta)\in S$. If $\int_a^bf<L_f(\Delta)$, then $x:=\frac{1}{2}\left(\int_a^bf+L_f(\Delta)\right)$ satisfies $\int_a^bf<x<L_f(\Delta)$, so $x$ is not an upper bound of $S$, which means $x\in A$. By the definition of cut point, $x\leq\int_a^bf<x$, a contradiction. Thus, $L_f(\Delta)\leq\int_a^bf$. On the other hand, if $U_f(\Delta)=\infty$, then $\int_a^bf\leq U_f(\Delta)$ holds immediately; otherwise $U_f(\Delta)\in\R$, and since $L_f(\Lambda)\leq U_f(\Delta)$ for any $\Lambda\in\mathcal{P}[a,b]$, we have $U_f(\Delta)\in B$, so by the definition of cut point, $\int_a^bf\leq U_f(\Delta)$. Consequently, $L_f(\Delta)\leq\int_a^bf\leq U_f(\Delta)$ for all $\Delta\in\mathcal{P}[a,b]$, as desired.

If $I\in\R$ also satisfies $L_f(\Delta)\leq I\leq U_f(\Delta)$ for all $\Delta\in\mathcal{P}[a,b]$, then $I$ is an upper bound of $S$, so $I\in B$, and is hence related to the cut point $\int_a^bf$ by the inequality $\int_a^bf\leq I$. If $\int_a^bf<I$, then $I-\int_a^bf>0$, and by \DIT, there exists a partition $\overline{\Delta}$ of $[a,b]$ such that $U_f(\overline{\Delta})-L_f(\overline{\Delta})<I-\int_a^bf$. But by the definition of $I$, we have $L_f(\overline{\Delta})\leq I\leq U_f(\overline{\Delta})$, which, by routine manipulations, implies $I-\int_a^bf\leq U_f(\overline{\Delta})-L_f(\overline{\Delta})$. Having obtained contradicting inequalities for $I-\int_a^bf$ and $U_f(\overline{\Delta})-L_f(\overline{\Delta})$, we conclude that $\int_a^bf= I$.
\end{proof}

\section{Compactness Arguments using Suprema}\label{SupSec}

Should the reader prefer \ES\ as the starting postulate for real function theory, the compactness arguments of the previous section admit a streamlined reformulation in terms of least upper bounds. Instead of constructing an explicit partition $(A,B)$ of the real line, one isolates the subset of $[a,b]$ on which the desired condition holds and its supremum plays an important role in the proof. The statements to be established are reproduced below.

\begin{enumerate}
    \item[\IVT]\emph{Intermediate Value Theorem.} If $f$ is continuous on $[a,b]$ and\linebreak $f(a)<0<f(b)$, then there exists $c\in[a,b]$ such that $f(c)=0$.
    \item[\DIT]\emph{Darboux Integrability Theorem.} If $f$ is continuous on $[a,b]$, then $f$ is \emph{Darboux integrable} of $[a,b]$: for each $\varepsilon>0$, there exists a partition $\Delta$ of $[a,b]$ such that the difference of the upper and lower Darboux sums of $f$ over $[a,b]$ satisfies the inequality $U_f(\Delta)-L_f(\Delta)<\varepsilon$.
    \item[\IET]\emph{Darboux Integral Theorem.} If $f$ is continuous on $[a,b]$, then $f$ \emph{has a Darboux integral} over $[a,b]$: there exists a unique $\int_a^bf\in\R$ such that for any partition $\Delta$ of $[a,b]$, we have $L_f(\Delta)\leq \int_a^bf\leq U_f(\Delta)$.
\end{enumerate}

\begin{proof}[Proof of the \IVT\ using suprema.] The desired conclusion is equivalent to the assertion that if $f$ is never zero on $[a,b]$ and negative at $a$, then $f$ is not positive at $b$. Let
\[S := \{x\in[a,b]\,:\,t\in[a,x]\implies f(t)<0\}.\]
Using the continuity of $f$ at $a$ and the condition $f(a)<0$, by an epsilon-delta argument, there exists $\delta_1>0$ such that $a+\delta_1\in[a,b]$ and $t\in[a,a+\delta_1]$ implies $f(t)<0$. Thus, $a+\delta_1$ is an element of $S$, so $S$ is nonempty. Moreover, since $S\sub[a,b]$, the endpoint $b$ is an upper bound of $S$. By \ES, $c := \sup S$ exists in $\R$, and we have $a<a+\delta_1\leq c\leq b$.

By the continuity of $f$ at $c$, there exists $\delta_2>0$ such that $a<a+\delta_2<c$ and that $f$ has constant sign on $[a+\delta_2,c]$. But since $a+\delta_2<c$, the number $a+\delta_2$ is not an upper bound of $S$, so there exists $\xi\in S$ such that $a<a+\delta_2<\xi$. Since $c$ is an upper bound of $S$, we further have $a<a+\delta_2<\xi\leq c$. Since $\xi\in S$ and $a+\delta_2\in[a,\xi]$, we find that $f$ is negative on $[a,\xi]$. In particular, $f(a+\delta_2)<0$. However, we also have $a+\delta_2\in[a+\delta_2,c]$, and since $f$ is negative at $a+\delta_2$ and has constant sign on $[a+\delta_2,c]$, then $f$ is negative on $[a+\delta_2,c]$. Consequently, $f$ is negative on $[a,\xi]\cup[a+\delta_2,c]=[a,c]$. At this point, we have proven $c\in S$.

Tending towards a contradiction, suppose $c<b$. Since $c\in S$, we have $f(c)<0$, and the continuity of $f$ at $c$ may be used in an epsilon-delta argument to prove that there exists $\delta_3>0$ such that $a<c<c+\delta_3\leq b$ and that $f$ is negative on $[c,c+\delta_3]$. The condition $c\in S$ implies $f$ is negative on $[a,c]$. Hence, $f$ is negative on $[a,c]\cup[c,c+\delta_3]=[a,c+\delta_3]$. At this point, we have produced an element $c+\delta_3$ of $S$ that is bigger than the upper bound $c$, a contradiction. Hence, $c{\not<}b$. Since $a\leq c\leq b$, we further have $c=b$. At this point, we have proven that $b\in S$, so $f(b)<0$, as desired.
\end{proof}

\begin{proof}[Proof of the \DIT\ using suprema.] Let $\mathcal{P}[a,x]$ denote the set of all partitions of $[a,x]$, and define
\[S := \left\{ x\in[a,b] \,:\, \exists \Delta_x \in \mathcal{P}[a,x]\quad U_f(\Delta_x)-L_f(\Delta_x) < \frac{x-a}{b-a}\varepsilon\right\}.\]
Using the continuity of $f$ at $a$, for $\frac{\varepsilon}{4(b-a)}>0$, there exists $\delta_1>0$ such that $a<a+\delta_1\leq b$ and that $t\in[a,a+\delta_1]$ implies $|f(t)-f(a)|<\frac{\varepsilon}{4(b-a)}$. If we let $t_0:=a$ and $t_1:=\min\{a+\delta_1,b\}$, then the trivial partition $\Delta_1:=\{t_0,t_1\}$ of $[a,a+\delta_1]$ satisfies
\begin{eqnarray}
    U_f(\Delta_1)-L_f(\Delta_1) &=& \left(\sup_{t\in[a,t_1]}f(t)-\inf_{t\in[a,t_1]}f(t)\right)(t_1-a),\nonumber\\
    &\le& \frac{\varepsilon}{2(b-a)}(t_1-a) < \frac{t_1-a}{b-a}\varepsilon.\nonumber
\end{eqnarray}
Thus, $t_1 \in S$, and the set $S$ is nonempty. Moreover, since $S\subseteq[a,b]$, the endpoint $b$ is an upper bound of $S$. By \ES, $c := \sup S$ exists in $\R$, and we further have $a<t_1\leq c\leq b$.

By the continuity of $f$ at $c$, for $\frac{\varepsilon}{4(b-a)}>0$, there exists $\delta_2>0$ such that $a < a+\delta_2 < c$ and $|f(t)-f(c)| < \frac{\varepsilon}{4(b-a)}$ for all $t \in [a+\delta_2,c]$. This implies that if $M_1:=\sup_{t\in [a+\delta_2,c]}f(t)$ and $m_1:=\inf_{t\in [a+\delta_2,c]}f(t) $, then $M_1-m_1<\frac{\varepsilon}{b-a}$. Since $a+\delta_2\in S$, there exists a partition $\Delta_2$ of $[a,a+\delta_2]$ such that $U_f(\Delta_2)-L_f(\Delta_2)<\frac{\delta_2}{b-a}\varepsilon$. Hence, the partition $\Delta_2\cup\{c\}$ of $[a,c]$ satisfies
\begin{eqnarray}
U_f(\Delta_2\cup\{c\})-L_f(\Delta_2\cup\{c\}) &=& U_f(\Delta_2)-L_f(\Delta_2)\nonumber\\
&&+ (M_1-m_1)(c-a-\delta_2),\nonumber \\
&<& \frac{\delta_2}{b-a}\varepsilon + \frac{\varepsilon}{b-a}(c-a-\delta_2)=\frac{c-a}{b-a}\varepsilon.\nonumber
\end{eqnarray}
At this point, we have proven that $c\in S$.

Tending towards a contradiction, suppose $c<b$. By the continuity of $f$ at $c$, there exists $\delta_3>0$ such that $a<c<c+\delta_3\le b$ and if $M_2:=\sup_{t\in [c,c+\delta_3]}f(t)$ and $m_2:=\inf_{t\in [c,c+\delta_3]}f(t)$, then $M_2-m_2<\frac{\varepsilon}{b-a}$. Considering the partition $\Delta_3:= \Delta_2\cup\{c,c+\delta_3\}$ of $[a, c+\delta_3]$, we obtain
\begin{eqnarray}
U_f(\Delta_3)-L_f(\Delta_3) &=& U_f(\Delta_2\cup\{c\})-L_f(\Delta_2\cup\{c\})+ (M_2-m_2)\delta_3,\nonumber \\
&<& \frac{c-a}{b-a}\varepsilon + \frac{\varepsilon}{b-a}\delta_3=\frac{(c+\delta_3)-a}{b-a}\varepsilon.\nonumber
\end{eqnarray}
At this point, we have produced an element $c+\delta_3$ of $S$ that is strictly greater than the upper bound $c$, a contradiction. Hence, $c \not< b$. Since $a \le c \le b$, we further have $c = b$. Consequently, $b \in S$, which implies the desired conclusion.
\end{proof}

\begin{proof}[Proof of $\DIT\implies\IET$ using suprema.] If $f$ is continuous on $[a,b]$, then by \DIT, there exists a partition $\Omega$ of $[a,b]$ such that $U_f(\Omega)-L_f(\Omega)<1$, which, if one of $U_f(\Omega)=\infty$ or $L_f(\Omega)=-\infty$ is true, leads to the contradiction $\infty<1$. Thus, $L_f(\Omega)\leq U_f(\Omega)<\infty$ and $-\infty<L_f(\Omega)$, and we now find that $L_f(\Omega)
\in\R$, so the set $S=\{L_f(\Delta)\in\R\,:\,\Delta\in\mathcal{P}[a,b]\}$, where $\mathcal{P}[a,b]$ is the set of all partitions of $[a,b]$, is a nonempty subset of $\R$. Since $L_f(\Lambda)\leq U_f(\Delta)$ for any $\Lambda,\Delta\in\mathcal{P}[a,b]$, we find that $U_f(\{a,b\})$ is an upper bound of $S$. By \ES, there exists $\int_a^bf\in\R$ such that $\int_a^bf=\sup S$. Given $\Delta\in\mathcal{P}[a,b]$, since $L_f(\Delta)$ is an element, and $U_f(\Delta)$ is an upper bound, of $S$, we have $L_f(\Delta)\leq\int_a^bf\leq U_f(\Delta)$ for all $\Delta\in\mathcal{P}[a,b]$, as desired. If $I\in\R$ also satisfies $L_f(\Delta)\leq I\leq U_f(\Delta)$ for all $\Delta\in\mathcal{P}[a,b]$, then, $I$ is an upper bound of $S$, and is hence related to the supremum $\int_a^bf$ by the inequality $\int_a^bf\leq I$. If $\int_a^bf<I$, then $I-\int_a^bf>0$, and by \DIT, there exists a partition $\overline{\Delta}$ of $[a,b]$ such that $U_f(\overline{\Delta})-L_f(\overline{\Delta})<I-\int_a^bf$. But by the definition of $I$, we have $L_f(\overline{\Delta})\leq I\leq U_f(\overline{\Delta})$, which, by routine manipulations, implies $I-\int_a^bf\leq U_f(\overline{\Delta})-L_f(\overline{\Delta})$. Having obtained contradicting inequalities for $I-\int_a^bf$ and $U_f(\overline{\Delta})-L_f(\overline{\Delta})$, we conclude that $\int_a^bf= I$.
\end{proof}

\section{Sequential Compactness Arguments}\label{SeqSec}

Finally, we consider \BWP\ as our foundational completeness axiom. Here, the bisection procedure of Section~\ref{NestedSec} is recast in terms of bounded sequences and their convergent subsequences, avoiding the explicit intersection of nested intervals while retaining the sequential flavor of elementary analysis. We restate the three theorems to be proven by this method.

\begin{enumerate}
    \item[\IVT]\emph{Intermediate Value Theorem.} If $f$ is continuous on $[a,b]$ and\linebreak $f(a)<0<f(b)$, then there exists $c\in[a,b]$ such that $f(c)=0$.
    \item[\DIT]\emph{Darboux Integrability Theorem.} If $f$ is continuous on $[a,b]$, then $f$ is \emph{Darboux integrable} of $[a,b]$: for each $\varepsilon>0$, there exists a partition $\Delta$ of $[a,b]$ such that the difference of the upper and lower Darboux sums of $f$ over $[a,b]$ satisfy the inequality $U_f(\Delta)-L_f(\Delta)<\varepsilon$.
    \item[\IET]\emph{Darboux Integral Theorem.} If $f$ is continuous on $[a,b]$, then $f$ \emph{has a Darboux integral} over $[a,b]$: there exists a unique $\int_a^bf\in\R$ such that for any partition $\Delta$ of $[a,b]$, we have $L_f(\Delta)\leq \int_a^bf\leq U_f(\Delta)$.
\end{enumerate}

\begin{proof}[Proof of the \IVT\ using the Bolzano-Weierstrass property.]
Let $a_1:=a$\linebreak and $b_1:=b$, so by assumption $f(a_1)<0<f(b_1)$. If $f\left(\frac{a_1+b_1}{2}\right)=0$, then $c:=\frac{a_1+b_1}{2}\in[a,b]$ satisfies $f(c)=0$, and we are done. Otherwise, suppose that for some $n\in\N$, real numbers $a_1,a_2,\ldots,a_n$ and $b_1,b_2,\ldots,b_n$, all in the interval $[a,b]$, have been determined such that $f(a_k)<0<f(b_k)$ and $b_k-a_k=2(b-a)\frac{1}{2^{k}}$ for all $k\in\{1,2,\ldots,n\}$. If $f\left(\frac{a_n+b_n}{2}\right)=0$, then $c:=\frac{a_n+b_n}{2}\in[a,b]$ satisfies $f(c)=0$, and we are done. Otherwise, we proceed with the induction. If $f\left(\frac{a_n+b_n}{2}\right)<0$, we set $a_{n+1}:=\frac{a_n+b_n}{2}$ and $b_{n+1}:=b_n$. If $f\left(\frac{a_n+b_n}{2}\right)>0$, we set $a_{n+1}:=a_n$ and $b_{n+1}:=\frac{a_n+b_n}{2}$. In both cases, $f(a_{n+1})<0<f(b_{n+1})$ and $b_{n+1}-a_{n+1}=\frac{b_n-a_n}{2}=2(b-a)\frac{1}{2^{n+1}}$. By induction, we have produced sequences $(a_n)$ and $(b_n)$, which are both sequences in $[a,b]$, and are hence bounded sequences, such that $f(a_n)<0<f(b_n)$ and $b_n-a_n=2(b-a)\frac{1}{2^{n}}$ for all $n\in\N$. By the \BWP, there exist convergent subsequences $(a_{N_k})$ and $(b_{N_k})$ of $(a_n)$ and $(b_n)$, respectively. By \AP, $\nseqLim(b_n-a_n)=2(b-a)\nseqLim\frac{1}{2^{n}}=0$, and since $(b_{N_k}-a_{N_k})$ is a subsequence of $(b_n-a_n)$, we have $\kseqLim(b_{N_k}-a_{N_k})=0$. Since $(a_{N_k})$ is convergent, by the addition rule for sequence limits,
\begin{eqnarray}
    c &:=&\kseqLim b_{N_k}=\kseqLim((b_{N_k}-a_{N_k})+a_{N_k})=\kseqLim(b_{N_k}-a_{N_k})+\kseqLim a_{N_k},\nonumber\\
    &=&0+\kseqLim a_{N_k}.\nonumber
\end{eqnarray}
Since $(a_{N_k})$ is a sequence in $[a,b]$, the condition $a\leq a_{N_k}\leq b$ for all $k\in\N$ implies $a\leq c=\kseqLim a_{N_k}\leq b$, so $c\in[a,b]$.

 Since $(a_{N_k})$ and $(b_{N_k})$ are convergent, by the continuity of $f$, so are the sequences $(f(a_{N_k}))$ and $(f(b_{N_k}))$. The condition $f(a_n)<0<f(b_n)$ applied to the subsequences $(a_{N_k})$ and $(b_{N_k})$ yields $f(a_{N_k})<0<f(b_{N_k})$, which implies $f(a_{N_k})\leq 0\leq f(b_{N_k})$ for all $k\in\N$. By the monotonicity of sequence limits, $\kseqLim f(a_{N_k})\leq 0\leq \kseqLim f(b_{N_k})$. Using the continuity of $f$, we further have $f\left(\kseqLim a_{N_k}\right)\leq 0\leq f\left(\kseqLim b_{N_k}\right)$, where, as proven earlier, both of the sequences $(a_{N_k})$ and $(b_{N_k})$ converge to $c\in[a,b]$. Thus, $f\left(c\right)\leq 0\leq f\left(c\right)$, and therefore, $f(c)=0$.
\end{proof}

\begin{proof}[Proof of the \DIT\ using the Bolzano-Weierstrass property.]
Given $\varepsilon>0$, suppose, for contradiction, that there exists no partition $\Delta$ of $[a,b]$ such that $U_f(\Delta)-L_f(\Delta)<\varepsilon$. Let $a_1:=a$ and $b_1:=b$. Suppose that for some $n\in\N$, real numbers $a_1,a_2,\ldots,a_n$ and $b_1,b_2,\ldots,b_n$, all in the interval $[a,b]$, have been determined such that $b_k-a_k=2(b-a)\frac{1}{2^{k}}$ and no partition $\Delta_k$ of $[a_k,b_k]$ satisfies $U_f(\Delta_k)-L_f(\Delta_k)<\frac{\varepsilon}{2^{k-1}}$ for all $k\in\{1,2,\ldots,n\}$. If there existed partitions $\Delta_1$ of $\left[a_n,\frac{a_n+b_n}{2}\right]$ and $\Delta_2$ of $\left[\frac{a_n+b_n}{2},b_n\right]$ such that $U_f(\Delta_1)-L_f(\Delta_1)<\frac{\varepsilon}{2^n}$ and $U_f(\Delta_2)-L_f(\Delta_2)<\frac{\varepsilon}{2^n}$, then $\Delta:=\Delta_1\cup\Delta_2$ would be a partition of $[a_n,b_n]$ satisfying
\[ U_f(\Delta)-L_f(\Delta) = \big(U_f(\Delta_1)-L_f(\Delta_1)\big) + \big(U_f(\Delta_2)-L_f(\Delta_2)\big) < \frac{\varepsilon}{2^n} + \frac{\varepsilon}{2^n} = \frac{\varepsilon}{2^{n-1}}, \]
contradicting our hypothesis on $[a_n,b_n]$. Thus, we define $[a_{n+1},b_{n+1}]$ to be whichever of $\left[a_n,\frac{a_n+b_n}{2}\right]$ or $\left[\frac{a_n+b_n}{2},b_n\right]$ has no partition $\Lambda$ that satisfies $U_f(\Lambda)-L_f(\Lambda)<\frac{\varepsilon}{2^n}$. In both cases, $b_{n+1}-a_{n+1}=\frac{b_n-a_n}{2}=2(b-a)\frac{1}{2^{n+1}}$, and no partition $\Delta$ of $[a_{n+1},b_{n+1}]$ satisfies the inequality $U_f(\Delta)-L_f(\Delta)<\frac{\varepsilon}{2^n}$.

By induction, we obtain sequences $(a_n)$ and $(b_n)$, which are sequences in $[a,b]$ and are hence bounded sequences, such that $b_n-a_n=2(b-a)\frac{1}{2^{n}}$ and no partition $\Delta_n$ of $[a_n,b_n]$ satisfies the inequality $U_f(\Delta_n)-L_f(\Delta_n)<\frac{\varepsilon}{2^{n-1}}$ for all $n\in\N$. By the \BWP, there exist convergent subsequences $(a_{N_k})$ and $(b_{N_k})$ of $(a_n)$ and $(b_n)$, respectively. 

By \AP, $\nseqLim(b_n-a_n)=2(b-a)\nseqLim\frac{1}{2^{n}}=0$, and since $(b_{N_k}-a_{N_k})$ is a subsequence of $(b_n-a_n)$, we have $\kseqLim(b_{N_k}-a_{N_k})=0$. Since $(a_{N_k})$ is convergent, by the addition rule for sequence limits,
\begin{eqnarray}
    c &:=&\kseqLim b_{N_k}=\kseqLim((b_{N_k}-a_{N_k})+a_{N_k})=\kseqLim(b_{N_k}-a_{N_k})+\kseqLim a_{N_k},\nonumber\\
    &=&0+\kseqLim a_{N_k}.\nonumber
\end{eqnarray}
Since $(a_{N_k})$ is a sequence in $[a,b]$, the condition $a\leq a_{N_k}\leq b$ for all $k\in\N$ implies $a\leq c=\kseqLim a_{N_k}\leq b$, so $c\in[a,b]$.

Since $f$ is continuous on $[a,b]$, $f$ is continuous at $c$, and since\linebreak  $\frac{\varepsilon}{4(b-a)}>0$, there exists $\delta>0$ such that if $x\in[a,b]$ satisfies $|x-c|<\delta$, then $|f(x)-f(c)|<\frac{\varepsilon}{4(b-a)}$. Since $\kseqLim a_{N_k}=c=\kseqLim b_{N_k}$ and $\delta>0$, by the definition of sequence convergence, there exist $K_1,K_2\in\N$ such that $|a_{N_k}-c|<\delta$ for all $k\geq K_1$ and $|b_{N_k}-c|<\delta$ for all $k\geq K_2$. Setting $K:=\max\{K_1,K_2\}\in\N$, we have $|a_{N_K}-c|<\delta$ and $|b_{N_K}-c|<\delta$, which implies $c-\delta<a_{N_K}$ and $b_{N_K}<c+\delta$. Thus, for any $x\in[a_{N_K},b_{N_K}]$, we have $c-\delta<a_{N_K}\leq x\leq b_{N_K}<c+\delta$, so $|x-c|<\delta$ and hence $|f(x)-f(c)|<\frac{\varepsilon}{4(b-a)}$. This further implies $f(c)-\frac{\varepsilon}{4(b-a)}\leq f(x)\leq f(c)+\frac{\varepsilon}{4(b-a)}$, where $f(x)$ is an arbitrary element of the image set $\{f(x)\,:\,x\in[a_{N_K},b_{N_K}]\}$. Thus, $f(c)-\frac{\varepsilon}{4(b-a)}$ is a lower bound, and $f(c)+\frac{\varepsilon}{4(b-a)}$ is an upper bound of said image set. Consequently, $f(c)-\frac{\varepsilon}{4(b-a)}\leq \inf_{x\in[a_{N_K},b_{N_K}]}f(x)$ and $\sup_{x\in[a_{N_K},b_{N_K}]}f(x)\le f(c)+\frac{\varepsilon}{4(b-a)}$, from which we obtain
\begin{eqnarray}
    \omega&:=&\sup_{x\in[a_{N_K},b_{N_K}]}f(x)-\inf_{x\in[a_{N_K},b_{N_K}]}f(x) ,\nonumber\\
    &\leq&  f(c)+\frac{\varepsilon}{4(b-a)}-f(c)+\frac{\varepsilon}{4(b-a)}=\frac{\varepsilon}{2(b-a)}.\nonumber
\end{eqnarray}

Now consider the trivial partition $\Delta_{N_K}:=\{a_{N_K},b_{N_K}\}$ of $[a_{N_K},b_{N_K}]$. The difference of upper and lower Darboux sums for $\Delta_{N_K}$ satisfies
\begin{eqnarray}
  U_f(\Delta_{N_K})-L_f(\Delta_{N_K}) &=&   \omega(b_{N_K}-a_{N_K}),\nonumber\\
  &\le& \frac{\varepsilon}{2(b-a)} \cdot 2(b-a)\frac{1}{2^{N_K}}=\frac{1}{2}\cdot\frac{\varepsilon}{2^{N_K-1}}<\frac{\varepsilon}{2^{N_K-1}},\nonumber
\end{eqnarray}
contradicting the condition that no partition of $[a_{N_K},b_{N_K}]$ satisfies the inequality $U_f(\Delta_{N_K})-L_f(\Delta_{N_K})<\frac{\varepsilon}{2^{N_K-1}}$. Therefore, there exists a partition $\Delta$ of $[a,b]$ satisfying $U_f(\Delta)-L_f(\Delta)<\varepsilon$.
\end{proof}

\begin{proof}[Proof of $\DIT\implies\IET$ using the Bolzano-Weierstrass property.] If $f$ is continuous on $[a,b]$, then by \DIT, there exists a partition $\Delta_1$ of $[a,b]$ such that $U_f(\Delta_1)-L_f(\Delta_1)<1$, which, if one of $U_f(\Delta_1)=\infty$ or $L_f(\Delta_1)=-\infty$ is true, leads to the contradiction $\infty<1$. Thus, $-\infty<L_f(\Delta_1)\leq U_f(\Delta_1)<\infty$, so $L_f(\Delta_1),U_f(\Delta_1)\in\R$. Suppose that for some $n\in\N$, partitions $\Delta_1\sub\Delta_2\sub\cdots\sub\Delta_n$ of $[a,b]$ have been identified such that $U_f(\Delta_k)-L_f(\Delta_k)<\frac{1}{k}$ for all $k\in\{1,2,\ldots,n\}$. Using \DIT\ again, there exists a partition $\Omega$ of $[a,b]$ such that $U_f(\Omega)-L_f(\Omega)<\frac{1}{n+1}$. We define $\Delta_{n+1}:=\Delta_n\cup\Omega$, which is also a partition of $[a,b]$, with $\Delta_n\sub\Delta_{n+1}$, and also, $\Omega\sub\Delta_{n+1}$, or that $\Delta_{n+1}$ is a refinement of $\Omega$, so $L_f(\Omega)\leq L_f(\Delta_{n+1})\leq U_f(\Delta_{n+1})\leq U_f(\Omega)$. By routine manipulation of inequalities,\linebreak $U_f(\Delta_{n+1})-L_f(\Delta_{n+1})\leq U_f(\Omega)-L_f(\Omega)<\frac{1}{n+1}$. By induction, we have produced a sequence $(\Delta_n)$ of partitions of $[a,b]$ such that\linebreak $0\leq U_f(\Delta_n)-L_f(\Delta_n)<\frac{1}{n}$ and $\Delta_n\sub\Delta_{n+1}$ for all $n\in\N$. The latter condition implies that for every $n\in\N$, $\Delta_n$ is a refinement of $\Delta_1$, and hence $L_f(\Delta_1)\leq L_f(\Delta_n)\leq U_f(\Delta_n)\leq U_f(\Delta_1)$. Thus, $(L_f(\Delta_n))$ is a sequence in $[L_f(\Delta_1),U_f(\Delta_1)]$, and is hence a bounded sequence.

By the \BWP, there exists a convergent subsequence $(L_f(\Delta_{N_k}))$ of $(L_f(\Delta_n))$. Let $\int_a^b f:=\kseqLim L_f(\Delta_{N_k})\in\R$. Using \AP, we obtain\linebreak $0\leq \nseqLim(U_f(\Delta_n)-L_f(\Delta_n))\leq \nseqLim\frac{1}{n}=0$, so $\nseqLim(U_f(\Delta_n)-L_f(\Delta_n))=0$. Since $(U_f(\Delta_{N_k})-L_f(\Delta_{N_k}))$ is a subsequence of $(U_f(\Delta_n)-L_f(\Delta_n))$, we have $\kseqLim(U_f(\Delta_{N_k})-L_f(\Delta_{N_k}))=0$. Since $(L_f(\Delta_{N_k}))$ is convergent, by the addition rule for sequence limits,
\begin{eqnarray}
\kseqLim U_f(\Delta_{N_k}) &=& \kseqLim\big((U_f(\Delta_{N_k})-L_f(\Delta_{N_k}))+L_f(\Delta_{N_k})\big),\nonumber\\
&=& \kseqLim(U_f(\Delta_{N_k})-L_f(\Delta_{N_k}))+\kseqLim L_f(\Delta_{N_k}),\nonumber\\
&=&0+\int_a^b f=\int_a^b f.\nonumber
\end{eqnarray}
Given an arbitrary partition $\Delta$ of $[a,b]$, for each $k\in\N$, the partition $\Delta\cup\Delta_{N_k}$ is a refinement of both $\Delta$ and $\Delta_{N_k}$. That is, $\Delta\sub\Delta\cup\Delta_{N_k}$ and $\Delta_{N_k}\sub\Delta\cup\Delta_{N_k}$, so
\begin{eqnarray}
L_f(\Delta) \leq L_f(\Delta\cup\Delta_{N_k}) &\leq& U_f(\Delta\cup\Delta_{N_k}) \leq U_f(\Delta_{N_k}),\nonumber\\
L_f(\Delta_{N_k}) \leq L_f(\Delta\cup\Delta_{N_k}) &\leq& U_f(\Delta\cup\Delta_{N_k}) \leq U_f(\Delta),\nonumber
\end{eqnarray}
which imply $L_f(\Delta)\leq U_f(\Delta_{N_k})$ and $L_f(\Delta_{N_k})\leq U_f(\Delta)$, respectively. By the monotonicity of sequence limits, we further obtain
\[
L_f(\Delta) \leq \kseqLim U_f(\Delta_{N_k}) = \int_a^b f = \kseqLim L_f(\Delta_{N_k}) \leq U_f(\Delta).
\]
If $I\in\R$ also satisfies $L_f(\Delta)\leq I\leq U_f(\Delta)$ for every partition $\Delta$ of $[a,b]$, then, in particular, taking $\Delta=\Delta_{N_k}$ for each $k\in\N$ gives $L_f(\Delta_{N_k})\leq I\leq U_f(\Delta_{N_k})$, and by the Squeeze Theorem for sequence limits, we obtain $I=\int_a^b f$.
\end{proof}

\section{The First Fundamental Theorem of Calculus}\label{FTCsec}

In this section, we shall prove a chain of implications involving the following real analysis principles, and perhaps shed more light into the elementary theory of integration.

\begin{enumerate}

    \item[\IET]\emph{Darboux Integral Theorem.} If $f$ is continuous on $[a,b]$, then $f$ \emph{has a Darboux integral} over $[a,b]$: there exists a unique $\int_a^bf\in\R$ such that for any partition $\Delta$ of $[a,b]$, we have $L_f(\Delta)\leq \int_a^bf\leq U_f(\Delta)$.
\item[\iFTC]\emph{Fundamental Theorem of Calculus, First Form.} If $f$ is continuous on $[a,b]$, then the rule of assignment $x\mapsto\int_a^xf(t)\,dt$ is an antiderivative of $f$ on $[a,b]$.
    \item[\DIT]\emph{Darboux Integrability Theorem.} If $f$ is continuous on $[a,b]$, then $f$ is \emph{Darboux integrable} of $[a,b]$: for each $\varepsilon>0$, there exists a partition $\Delta$ of $[a,b]$ such that the difference of the upper and lower Darboux sums of $f$ over $[a,b]$ satisfy the inequality $U_f(\Delta)-L_f(\Delta)<\varepsilon$.
    \item[\BVT]\emph{Bounded Value Theorem.} If $f$ is continuous on $[a,b]$, then there exists $M\in\R$ such that $x\in[a,b]$ implies $f(x)\leq M$.
    \end{enumerate}

    \begin{proof}[Proof of $\IET\implies\iFTC$.] The continuity of $f$ on $[a,b]$ implies the continuity of $f$ on $[a,x]\sub[a,b]$ for any $x\in[a,b]$. By \IET, $\int_a^xf(t)\,dt$ is defined for all $x\in[a,b]$. By the definition of a function that has an integral over $[a,x]$, the integral is unique. Henceforth, the rule of assignment $x\mapsto\int_a^xf(t)\,dt$ is a function $[a,b]\into\R$. Given $c\in[a,b]$ and $\varepsilon>0$, by the continuity of $f$ at $c$, there exists $\delta>0$ such that for all $x,t\in[a,b]$, if $|x-t|<\delta$, then $|f(x)-f(t)|<\frac{\varepsilon}{2}$. Thus, for all $x$ with $0<|x-c|<\delta$, we have 
\begin{eqnarray}
\left|\frac{\int_a^xf(t)\,dt-\int_a^cf(t)\,dt}{x-c}-f(c)\right| &=&\frac{1}{|x-c|}\left|\int_{\min\{x,c\}}^{\max\{x,c\}}(f(t)-f(c))\,dt\right|,\nonumber\\
&\leq&\frac{1}{|x-c|}\int_{\min\{x,c\}}^{\max\{x,c\}}|f(t)-f(c)|\,dt,\nonumber\\
&\leq&\frac{1}{|x-c|}\int_{\min\{x,c\}}^{\max\{x,c\}}\frac{\varepsilon}{2}\,dt,\nonumber\\
&&=\frac{1}{2|x-c|}\varepsilon\cdot|x-c|<\varepsilon.\nonumber
\end{eqnarray}
This proves that the function $x\mapsto\int_a^xf(t)\,dt$ is differentiable at $c$, with the derivative at $c$ equal to $f(c)$, as desired.
\end{proof}

\begin{proof}[Proof of $\iFTC\implies\DIT$.] Since $f$ is continuous on $[a,b]$, by \iFTC, $f$ has a Darboux integral $\int_a^bf$ over $[a,b]$. Given $\varepsilon>0$, the number $I-\frac{\varepsilon}{2}$ (respectively, $I+\frac{\varepsilon}{2}$) is distinct from $I$, so there exists a partition $\Delta_1$ (respectively, $\Delta_2$) of $[a,b]$ such that $L_f(\Delta_1){\not\leq}I-\frac{\varepsilon}{2}$ (respectively, $I+\frac{\varepsilon}{2}{\not\leq}U_f(\Delta_2)$). Since $\Delta_1\cup\Delta_2$ is a refinement of both $\Delta_1$ and $\Delta_2$, we obtain
\[I-\frac{\varepsilon}{2}<L_f(\Delta_1)\leq L_f(\Delta_1\cup\Delta_2)\leq I\leq U_f(\Delta_1\cup\Delta_2)\leq U_f(\Delta_2)<I+\frac{\varepsilon}{2},\]
and by routine manipulation, we further have $U_f(\Delta_1\cup\Delta_2)-L_f(\Delta_1\cup\Delta_2)<(I+\frac{\varepsilon}{2})-(I-\frac{\varepsilon}{2})=\varepsilon$. Hence, $f$ is Darboux integrable over $[a,b]$.
\end{proof}

\begin{proof}[Proof of $\DIT\implies\BVT$.]
Since $1>0$, there exists a partition $\Delta$ of $[a,b]$. Ordering the elements of $\Delta$ as $a=x_0<x_1<\cdots<x_n=b$, we have
\[ U_f(\Delta)-L_f(\Delta)=\sum_{k=1}^n\left(\sup_{x\in[x_{k-1},x_k]}f(x)-\inf_{x\in[x_{k-1},x_k]} f(x)\right)(x_k-x_{k-1})<1. \]
For each $k\in\{1,2,\ldots,n\}$, since $\sup_{x\in[x_{k-1},x_k]}f(x) \ge \inf_{x\in[x_{k-1},x_k]} f(x)$ and $x_k-x_{k-1}>0$, every term in the sum is non-negative. Consequently, each individual term satisfies
\[ \left(\sup_{x\in[x_{k-1},x_k]}f(x)-\inf_{x\in[x_{k-1},x_k]} f(x)\right)(x_k-x_{k-1}) \le U_f(\Delta)-L_f(\Delta) < 1. \]
Dividing by the positive number $x_k-x_{k-1}$ yields
\[ \sup_{x\in[x_{k-1},x_k]}f(x)-\inf_{x\in[x_{k-1},x_k]} f(x) < \frac{1}{x_k-x_{k-1}}. \]
For each $k\in\{1,2,\ldots,n\}$, fix $c_k\in[x_{k-1},x_k]$. Since $\inf_{x\in[x_{k-1},x_k]} f(x) \leq f(c_k)$, 
\[ \sup_{x\in[x_{k-1},x_k]}f(x) < \inf_{x\in[x_{k-1},x_k]} f(x) + \frac{1}{x_k-x_{k-1}} \leq f(c_k) + \frac{1}{x_k-x_{k-1}}. \]
If we define
\[ M := \max \left\{ f(c_k) + \frac{1}{x_k-x_{k-1}} \,:\, k\in\{1,2,\ldots,n\} \right\}, \]
then $M$ is the maximum of finitely many real numbers. Given $t\in[a,b]=\bigcup_{k=1}^n[x_{k-1},x_k]$, there exists $K\in\{1,2,\ldots,n\}$ such that $t\in [x_{K-1},x_K]$, so $f(t)\leq \sup_{x\in[x_{K-1},x_K]}f(x) \le f(c_K) + \frac{1}{x_K-x_{K-1}} \le M$.
\end{proof}

\section{The Geometric Principles of Calculus}\label{GeomPrinceSec}

We now look into the real analysis principles that give us intuition on the ``shape'' of the graphs of functions based on their continuity or differentiability. These principles may be viewed to constitute an elementary theory of differentiation.

\begin{enumerate}
    \item[\EVT]\emph{Extreme Value Theorem.} If $f$ is continuous on $[a,b]$, then there exist $c_1,c_2\in[a,b]$ such that $x\in[a,b]$ implies $f(c_1)\leq f(x)\leq f(c_2)$.
    \item[\DT]\emph{Darboux's Theorem.} If $f$ is differentiable on $[a,b]$ and $f'(a)<0<f'(b)$, then there exists $c\in[a,b]$ such that $f'(c)=0$.
    \item[\RT]\emph{Rolle's Theorem.} If $f$ is continuous on $[a,b]$ and differentiable on $(a,b)$, and if $f(a)=0=f(b)$, then there exists $c\in(a,b)$ such that $f'(c)=0$.
    \item[\eMVT]\emph{Extended Mean Value Theorem.} If $f$ and $g$ are continuous on $[a,b]$ and differentiable on $(a,b)$, and if $g'$ is nonzero on $(a,b)$, then there exists $c\in(a,b)$ such that $\frac{f'(c)}{g'(c)}=\frac{f(b)-f(a)}{g(b)-g(a)}$.
    \item[\MVT]\emph{Mean Value Theorem.} If $f$ is continuous on $[a,b]$ and differentiable on $(a,b)$, then there exists $c\in(a,b)$ such that $f'(c)=\frac{f(b)-f(a)}{b-a}$.
    \item[\SIFT]\emph{Strictly Increasing Function Theorem.} If $f$ is differentiable on $(a,b)$ and $f'$ is positive on $(a,b)$, then $a<x<c<b$ implies $f(x)<f(c)$.
    \item[\IFT]\emph{Increasing Function Theorem.} If $f$ is differentiable on $(a,b)$ and $f'$ is nonnegative on $(a,b)$, then $a<x<c<b$ implies $f(x)\leq f(c)$.
    \item[\CFT]\emph{Constant Function Theorem.} If $f$ is differentiable on $(a,b)$ and $f'$ is zero on $(a,b)$, then $f$ is constant on $(a,b)$.
\end{enumerate}

The list of geometric principles of calculus in \cite[pp.~466--467]{loo74} consists of the above theorems except $\DT$ and $\RT$, and instead includes $\IVT$. The above list and ordering is in accordance with the proofs of the implications that follow.

\begin{proof}[Proof of $\ES\&\BVT\implies\EVT$.] Since $f$ is continuous on $[a,b]$, by\linebreak \BVT, there exists $M\in\R$ that is an upper bound of $\{f(x)\,:\,x\in[a,b]\}$, which is nonempty for it contains $f(a)$. By \ES, there exists $y\in\R$ such that $y:=\sup\{f(x)\,:\,x\in[a,b]\}$. In particular, $y$ is an upper bound, so $x\in[a,b]$ implies $f(x)\leq y$. If no $x\in[a,b]$ satisfies $f(x)=y$, then $f(x)<y$ for all $x\in[a,b]$. This means that $g(x)=\frac{1}{y-f(x)}$ is defined, and in fact positive, for all $x\in[a,b]$. By operations on continuous functions, $g$ is continuous on $[a,b]$, and \BVT\  applies to it, so there exists $\varepsilon\in\R$ such that $x\in[a,b]$ implies $0<g(x)=\frac{1}{y-f(x)}\leq \varepsilon$, which, by routine manipulations, implies $f(x)\leq y-\frac{1}{\varepsilon}<y$. Thus, we have produced an upper bound $y-\frac{1}{\varepsilon}$ of $\{f(x)\,:\,x\in[a,b]\}$ that is strictly less than the supremum $y$, a contradiction. Therefore, there exists $c\in[a,b]$ such that $f(c)=y\geq f(x)$ for all $x\in[a,b]$. 

The continuity of $f$ implies the continuity of $-f$ on $[a,b]$, so there exists $c_1\in[a,b]$ such that $x\in[a,b]$ implies $(-f)(c_1)\geq (-f)(x)$. Using $c_2:=c$, we have $f(c_1)\leq f(x)\leq f(c_2)$ for all $x\in[a,b]$.
\end{proof}

\begin{proof}[Proof of $\EVT\implies\DT$.] Since $f$ is differentiable on $[a,b]$, $f$ is continuous on $[a,b]$. By \EVT, there exists $c\in[a,b]$ such that $x\in[a,b]$ implies $f(c)\leq f(x)$. Since $f'(a)<0$, or that $-f'(a)>0$, by an epsilon-delta argument, there exists $\delta_1>0$ such that $a+\delta_1\in[a,b]$ and that $f(a)>f(a+\delta_1)$. That is, the implication ``$x\in[a,b]$ implies $f(a)\leq f(x)$,'' is false, so we cannot have $c=a$. Similarly, we may use $f'(b)>0$ in an epsilon-delta argument to produce a $\delta_2>0$ such that $b-\delta_2\in[a,b]$ with $f(b)>f(b-\delta_2)$, which makes the implication ``$x\in[a,b]$ implies $f(b)\leq f(x)$'' false, so $c\neq b$. Hence, $c\in(a,b)$, and the one-sided limits of the difference quotient $\frac{f(x)-f(c)}{x-c}$ may be routinely used to prove that $f'(c)=0$.
\end{proof}

\begin{proof}[Proof of $\EVT\implies\RT$.] Since $f$ is continuous on $[a,b]$, by \EVT, there exist $c_1,c_2\in[a,b]$ such that $x\in[a,b]$ implies $f(c_1)\leq f(x)\leq f(c_2)$. If $f(c_1)=f(c_2)$, then $f(x)=f(c_2)$ for all $x\in[a,b]$, or that $f$ is a constant function, which has a zero derivative, and we are done. Suppose henceforth that $f(c_1)<f(c_2)$. Thus, $f(c_1)\neq f(c_2)$, and since $f$ is a function, $c_1\neq c_2$. If $c_1$ and $c_2$ are both endpoints of $[a,b]$, then $c_1\neq c_2$ implies either $c_1=a$ and $c_2=b$ or $c_1=b$ and $c_2=a$. Hence, the assumption that $f(a)=0=f(b)$ leads to $f(c_1)=f(c_2)$ in both cases, and this contradicts $f(c_1)\neq f(c_2)$ from earlier. Thus, there exists $c\in\{c_1,c_2\}$ such that $c\in (a,b)$, which further satisfies either $x\in[a,b]$ implies $f(c)\leq f(x)$ or $x\in[a,b]$ implies $f(c)\geq f(x)$. For each case, the one-sided limits of the difference quotient $\frac{f(x)-f(c)}{x-c}$ may be routinely used to prove that $f'(c)=0$.
\end{proof}

\begin{proof}[Proof of $\RT\implies\eMVT$.] We first prove that $g(b)-g(a)\neq 0$. Otherwise, $g(b)=g(a)$ and $G(x):=g(x)-g(a)=g(x)-g(b)$ is continuous on $[a,b]$, differentiable on $(a,b)$, satisfies $G(a)=0=G(b)$ and has derivative $G'=g'$ on $(a,b)$. By \RT, there exists $k\in(a,b)$ such that $g'(k)=G'(k)=0$, contradicting the assumption that $g'$ is never zero on $(a,b)$. Henceforth, $g(b)-g(a)\neq 0$, and this now allows us to define $F(x):=f(x)-f(a)-\frac{f(b)-f(a)}{g(b)-g(a)}[g(x)-g(a)]$, where $F$ is continuous on $[a,b]$, differentiable on $(a,b)$ and satisfies $F(a)=0=F(b)$. By \RT, there exists $c\in(a,b)$ such that $0=F'(c)=f'(c)-\frac{f(b)-f(a)}{g(b)-g(a)}g'(c)$, which, because $g'(c)\neq 0$ by assumption, may be rearranged into $\frac{f'(c)}{g'(c)}=\frac{f(b)-f(a)}{g(b)-g(a)}$.
\end{proof}

\begin{proof}[Proof of $\eMVT\implies\MVT$.] If $g(x):=x$, then $g$ is continuous on $[a,b]$, differentiable on $(a,b)$, with $g'=1$, and hence nonzero on $(a,b)$. By \eMVT, there exists $c\in(a,b)$ such that $f'(c)=\frac{f'(c)}{1}=\frac{f(b)-f(a)}{b-a}$.
\end{proof}

\begin{proof}[Proof of $\MVT\implies\SIFT$.] If $a<x<c<b$, then $[x,c]\sub(a,b)$, and the differentiability of $f$ on $(a,b)$ implies the differentiability of $f$ on $[x,c]$. Since differentiability implies continuity, $f$ is continuous on $[x,c]$. Furthermore, the differentiability of $f$ on $[x,c]$ implies differentiability on $(x,c)\sub[x,c]$. By \MVT, there exists $k\in(x,c)$ such that $f'(k)=\frac{f(x)-f(c)}{x-c}$. Since $f'$ is assumed to be positive on $(a,b)$, and hence on $(x,c)\sub (a,b)$, we further have $0<\frac{f(x)-f(c)}{x-c}$, where $a<x<c<b$ implies $x-c$ is negative. Thus, $0\cdot(x-c)>\frac{f(x)-f(c)}{x-c}\cdot(x-c)$, and therefore, $f(x)<f(c)$.
\end{proof}

\begin{proof}[Proof of $\SIFT\implies\IFT$.] Let $\varepsilon>0$. The function $x\mapsto f(x)+\frac{\varepsilon}{b-a} x$ is continuous on $[a,b]$, differentiable on $(a,b)$, with derivative $x\mapsto f'(x)+\frac{\varepsilon}{b-a}$, which is positive on $(a,b)$ because $f'$ is nonnegative on $(a,b)$. If $a<x<c<b$, then $\frac{c-x}{b-a}<1$, and by \SIFT, $f(x)+\frac{\varepsilon}{b-a} x<f(c)+\frac{\varepsilon}{b-a} c$, which implies $f(x)<f(c)+\frac{c-x}{b-a}\varepsilon<f(c)+\varepsilon$. That is, $f(x)<f(c)+\varepsilon$ for any $\varepsilon>0$. Therefore, $f(x)\leq f(c)$.
\end{proof}

\begin{proof}[Proof of $\IFT\implies\CFT$.] If $f'=0$ on $(a,b)$, then $f'\geq 0$ and $f'\leq 0$ on $(a,b)$, where the latter implies $(-f)'=-f'\geq 0$. If $a<x<c<b$, then by \IFT, $f(x)\leq f(c)$ and $-f(x)\leq -f(c)$, which implies $f(x)\geq f(c)$. Thus, $f(x)=f(c)$ for any $x,c\in(a,b)$, or that $f$ is constant on $(a,b)$.
\end{proof}

\section{The Second Fundamental Theorem of Calculus}\label{iiFTCsec}

What remains in our naive development of real function theory is the sort of dual of \iFTC. Whereas $\iFTC$ deals with what may be informally viewed as the ``differentiation of an integral,'' the other fundamental theorem of calculus is about the ``integration of a derivative.'' We need $\DIT$  to establish the existence of the integral in question, but the main mechanism in the proof is \MVT.

\begin{enumerate}
    \item[\iiFTC]\emph{Fundamental Theorem of Calculus, Second Form.} If $f$ is continuous on $[a,b]$ and if $F$ is an antiderivative of $f$, then $\int_a^bf(t)\,dt=F(b)-F(a)$.
\end{enumerate}

\begin{proof}[Proof of $\DIT\&\MVT\implies\iiFTC$.] Since $f$ is continuous on $[a,b]$, by \DIT, $f$ is integrable over $[a,b]$ and hence has an integral over $[a,b]$. That is, there exists a unique real number $\int_a^bf$ such that for any partition $\Delta$ of $[a,b]$, we have $L_f(\Delta)\leq\int_a^bf\leq U_f(\Delta)$. Suppose that indeed $\Delta$ is an arbitrary partition of $[a,b]$ with elements listed as $a=x_0<x_1<\cdots<x_n=b$. Since $F$ is an antiderivative of $f$, it is differentiable, and hence continuous, on $[x_{k-1},x_k]\sub[a,b]$ for all $k\in\{1,2,\ldots,n\}$. The differentiability of $F$ on $[x_{k-1},x_k]$ implies its differentiability on $(x_{k-1},x_k)\sub[x_{k-1},x_k]$. By \MVT, for each $k\in\{1,2,\ldots,n\}$, there exists $c_k\in(x_{k-1},x_k)$ such that $F'(c_k)=\frac{F(x_k)-F(x_{k-1})}{x_k-x_{k-1}}$, which implies $F(x_k)-F(x_{k-1})=F'(c_k)\cdot(x_k-x_{k-1})=f(c_k)\cdot(x_k-x_{k-1})$, where the right-most member is at most $M_k(x_{k}-x_{k-1})$ and at least $m_k(x_{k}-x_{k-1})$, where $M_k=\sup\{f(x)\,:\,x\in[x_{k-1},x_k]\}$ and $m_k=\inf\{f(x)\,:\,x\in[x_{k-1},x_k]\}$. Thus,
\[m_k(x_k-x_{k-1})\leq F(x_k)-F(x_{k-1})\leq M_k(x_k-x_{k-1}), \]
for all $k\in\{1,2,\ldots,n\}$, and taking the summation of every member over all $k$, we obtain $L_f(\Delta)\leq F(b)-F(a)\leq U_f(\Delta)$. Since $\Delta$ is an arbitrary partition of $[a,b]$, by the uniqueness of $\int_a^bf$, we have $\int_a^bf=F(b)-F(a)$.
\end{proof}

\section{Summary of Naive Real Function Theory}\label{RealSumSec}

Having established the equivalence of our four candidate completeness axioms and deduced from them the central theorems of differentiation and integration, we are now in a position to view the architecture of naive real function theory as a whole. The logical dependencies among all the axioms and principles developed in Sections~\ref{AxiomSec}--\ref{iiFTCsec} are summarized in the following diagram.

\begin{center}
\begin{tikzpicture}[
    scale=0.5,             
    transform shape,       
    double arrow/.style={
        double, 
        double distance=1.5pt, 
        draw=#1, 
        ->, 
        >={Latex[length=4pt, width=5pt, color=#1]}, 
        shorten >= 3pt, 
        shorten <= 3pt,
        rounded corners=10pt
    },
    vertex/.style={
        draw, 
        rectangle, 
        rounded corners=5pt, 
        minimum width=40pt, 
        minimum height=20pt
    },
    wide box/.style={
        draw=red,
        fill=red!5,
        rectangle,
        rounded corners=5pt,
        minimum width=25.0cm,
        minimum height=1.35cm
    }
]

    \node[wide box] (TopBox) at (7.50, 18.87) {};

    \node[vertex, draw=blue, fill=blue!10] (Star_top) at (-3.85, 18.87) {\Star};
    \node at (-2.85, 18.87) {$\equiv$};
    \node[vertex, draw=red, fill=red!10]  (CA_top)   at (-1.85, 18.87) {\CA};
    \node[vertex, draw=red, fill=red!10]  (ES)       at (5.05, 18.87) {\ES};
    \node[vertex, draw=red, fill=red!10]  (NIP)      at (11.95, 18.87) {\AP\&\NIP};
    \node[vertex, draw=red, fill=red!10]  (BWP)      at (18.85, 18.87) {\BWP};

    \node[vertex, draw=black, fill=darkgray!10]        (IVT)    at (-3.85, 13.585) {\IVT};
    \node[vertex, draw=green!60!black, fill=green!10] (iFTC)   at (-0.26, 13.585) {\AP\&\iFTC};
    \node[vertex, draw=red, fill=red!10]              (DIT)    at (4.04, 13.585) {\AP\&\DIT};
    \node[vertex, draw=black, fill=darkgray!10]        (AP_BVT) at (8.25, 13.585) {\AP\&\BVT};
    \node[vertex, draw=green!60!black, fill=green!10] (BVT)    at (11.76, 13.585) {\BVT};
    \node[vertex, draw=green!60!black, fill=green!10] (iiES)   at (14.65, 16.23) {\ES};
    \node[vertex, draw=green!60!black, fill=green!10] (EVT)    at (14.65, 13.585) {\EVT};
    \node[vertex, draw=green!60!black, fill=green!10] (SIFT)   at (17.6, 6.95) {\SIFT};
    \node[vertex, draw=black, fill=darkgray!10]        (DT)     at (18.85, 13.585) {\DT};

    \node[vertex, draw=green!60!black, fill=green!10] (IET)  at (4.04, 9.28) {\AP\&\IET};
    \node[vertex, draw=blue, fill=blue!10]           (Star) at (4.04, 4.98) {\Star};

    \coordinate (Star_lane) at (-2.41, 0);

    \draw[double arrow=red, shorten >= 0pt, shorten <= 0pt]            (TopBox.south -| DIT) -- (DIT);
    \draw[double arrow=green!60!black, shorten >= 0pt, shorten <= 0pt] (iFTC) -- (DIT);
    \draw[double arrow=green!60!black, shorten >= 0pt, shorten <= 0pt] (DIT) -- (AP_BVT);
    \draw[double arrow=red, shorten >= 0pt, shorten <= 0pt]            (DIT) -- (IET);
    \draw[double arrow=green!60!black, shorten >= 0pt, shorten <= 0pt] (IET) -| (iFTC);
    \draw[double arrow=red, shorten >= 0pt, shorten <= 0pt]            (TopBox.south -| Star_lane) |- (Star);
    \draw[double arrow=red, shorten >= 0pt, shorten <= 0pt]            (Star) -- (IET);

    \node[vertex, draw=green!60!black, fill=green!10] (RT)    at (14.65, 11.37) {\RT};
    \node[vertex, draw=green!60!black, fill=green!10] (eMVT)  at (14.65, 9.16) {\eMVT};
    \node[vertex, draw=green!60!black, fill=green!10] (MVT)   at (14.65, 6.95) {\MVT};
    \node[vertex, draw=black, fill=darkgray!10]        (iiFTC) at (14.65, 4.74) {\iiFTC};

    \node[vertex, draw=green!60!black, fill=green!10] (IFT) at (17.6, 4.86) {\IFT};
    \node[vertex, draw=black, fill=darkgray!10]        (CFT) at (17.6, 2.77) {\CFT};

    \draw[double arrow=green!60!black, <->, shorten >= 0pt, shorten <= 0pt] (CA_top) -- (ES);
    \draw[double arrow=green!60!black, <->, shorten >= 0pt, shorten <= 0pt] (ES) -- (NIP);
    \draw[double arrow=green!60!black, <->, shorten >= 0pt, shorten <= 0pt] (NIP) -- (BWP);

    \draw[double arrow=red, shorten >= 0pt, shorten <= 0pt]            (TopBox.south -| IVT) -- (IVT);
    \draw[double arrow=green!60!black, shorten >= 0pt, shorten <= 0pt] (TopBox.south -| iiES) -- (iiES);

    \draw[double arrow=green!60!black, shorten >= 0pt, shorten <= 0pt] (AP_BVT) -- (BVT);
    \draw[double arrow=green!60!black, shorten >= 0pt, shorten <= 0pt] (iiES) -- (EVT);
    \draw[double arrow=green!60!black, shorten >= 0pt, shorten <= 0pt] (BVT) -- (EVT);
    \draw[double arrow=green!60!black, shorten >= 0pt, shorten <= 0pt] (MVT) -- (SIFT);
    \draw[double arrow=green!60!black, shorten >= 0pt, shorten <= 0pt] (MVT) -- (iiFTC);

    \draw[double arrow=green!60!black, shorten >= 0pt, shorten <= 0pt] (SIFT) -- (IFT);
    \draw[double arrow=green!60!black, shorten >= 0pt, shorten <= 0pt] (IFT) -- (CFT);

    \draw[double arrow=green!60!black, shorten <= 0pt, shorten >= 0pt] (EVT) -- (RT);
    \draw[double arrow=green!60!black, shorten <= 0pt, shorten >= 0pt] (EVT) -- (DT);
    \draw[double arrow=green!60!black, shorten <= 0pt, shorten >= 0pt] (RT) -- (eMVT);
    \draw[double arrow=green!60!black, shorten <= 0pt, shorten >= 0pt] (eMVT) -- (MVT);



\end{tikzpicture}
\end{center}

The red arrows represent implications that were proven using compactness arguments in Sections~\ref{NestedSec}--\ref{SeqSec}, while the arrows in green represent implications that may be proven using arguments that may be found in the usual real analysis literature\textemdash the specialized compactness arguments are not required and the argumentation in the proof naturally stems from the notions in the theory. The wide box at the very top of the diagram represents the completeness axiom that may be used as a starting point for a naive treatment of real function theory. As elaborated on in Section~\ref{AxiomSec}, we have a choice of one of four: \CA, \ES, \AP\&\NIP, or \BWP, and these were proven to be equivalent. If \Star\  is one of these four possible completeness axioms, then the third and final proof in each of Sections~\ref{NestedSec}--\ref{SeqSec} is a proof that \Star\&\DIT\  implies \IET. The nodes in gray are the terminal statements in this  treatment of real function theory in this exposition. These are \IVT, \iiFTC, \CFT, and \DT. The statement \AP\&\BVT\  is also considered as a terminal statement for a reason that will be apparent in the next section, which will also shed a light on why every statement in the chain
\[\IET\implies\iFTC\implies\DIT\implies\BVT,\]
is connected to \AP\  by conjunction. At this stage, a student who wishes simply to see how the major theorems of single-variable calculus rest upon a chosen completeness axiom has reached a natural resting point. 

\section{The Cut Axiom as a Necessary Condition}\label{DedekindSec}

 We have presented a treatment of real function theory that does not directly use topology (open sets and closed sets of real numbers, and continuous functions on these sets). The previous section should have been the end of this story. The student should have had good grasp of the basic theory of calculus, which is elementary real analysis, or elementary real function theory, with a naive treatment. This section is only for those who are interested to go deeper into the theory, but we have again taken a naive approach. We do not delve into the specifics of ordered fields the axiomatic way. We simply focus on the logic needed in the succeeding proofs by contraposition. To better present how all theorems we dealt with lead back to the completeness axioms, we introduce here the least possible amount of topological preliminaries. 

Given $S\sub\R$ and $c\in \R$, the assertion that there exists an open interval $(a,b)$ such that $c\in (a,b)\sub S$ is equivalent to the assertion that there exists $\delta>0$ such that $|x-c|<\delta$ implies $x\in S$. (For necessity, use $\delta:=\min\{b-c,c-a\}$, while for sufficiency, use $a=c-\delta$ and $b=c+\delta$.) In such a case, we say that $c$ is an \emph{interior point} of $S$. If every element of $S$ is an interior point, then $S$ is an \emph{open set}. All open intervals are open sets.

\begin{proposition}\label{OpenPastingProp} If $A$ and $B$ are nonempty disjoint open subsets of $\R$, if $F_1$ is continuous (respectively, differentiable) on $A$, and if $F_2$ is continuous (respectively, differentiable) on $B$, then the function $f:A\cup B\into\R$ defined by
\begin{eqnarray}
    f(x)=\begin{cases}
        F_1(x), &x\in A,\\
        F_2(x), &x\in B,
    \end{cases}\nonumber
\end{eqnarray}
is continuous (respectively, differentiable) on $A\cup B$, and so is the restriction of $f$ to any interval contained in $A\cup B$.
\end{proposition}
\begin{proof}[Proof (Sketch).] This may be proven by routine epsilon-delta arguments. Given $c\in A\cup B$ and given $\varepsilon>0$, the continuity or differentiability of $F_1$ or $F_2$ at $c$ may be used to produce a $\delta_1>0$ satisfying certain properties, and using the fact that $A$ (or $B$) is open, some $\delta_2>0$ may be produced satisfying certain conditions. The epsilon-delta proof for the continuity or differentiability of $f$ at $c$ may be obtained by using $\delta:=\min\{\delta_1,\delta_2\}$. The continuity or differentiability of $f$ at any element of $A\cup B$ holds at any element of a subset of $A\cup B$.
\end{proof}

\begin{example}\label{MainCounterExEx} If $A$ and $B$ are disjoint open subsets of $\R$, and if $a\in A$ and $b\in B$ with $[a,b]\sub A\cup B$, then the function $f:A\cup B\into\R$ defined by
\begin{eqnarray}
    f(x)=\begin{cases}
        -1, &x\in A\cap[a,b],\\
        1, &x\in B\cap[a,b],
    \end{cases}\label{MainCounterEx}
\end{eqnarray}
is continuous and differentiable on $[a,b]$. 
\end{example}

\begin{proposition} If $(A,B)$ is a cut of $\R$ such that no cut point is contained in $A$ (respectively, $B$), then $A$ (respectively, $B$) is an open set.
\end{proposition}
\begin{proof} Let $a\in A$. If $a$ is an upper bound of $A$, then $A\sub(-\infty,a]$ and by the definition of cut, any $b\in B$ satisfies $a<b$, so $B\sub[a,\infty)$. We have thus shown that $a$ is a cut point of $(A,B)$, contradicting the assumption that no cut point is in $A$. Hence, $a$ is not an upper bound of $A$, so there exists $u\in A$ such that $a<u$. Since $A$ is downward closed, $a\in(a-1,u)\sub(-\infty,u]\sub A$. Therefore, $A$ is an open set. The proof for $B$ is analogous.
\end{proof}

\begin{remark}\label{CArem} In an argument that requires the negation of $\CA$, there exists a cut $(A,B)$ of $\R$ that does not have a cut point. Hence, neither $A$ nor $B$ contains a cut point, and so, both $A$ and $B$ are open sets. If $a\in A$ and $b\in B$, then $[a,b]\sub \R=A\cup B$.
\end{remark}

We now proceed with the following \underline{proofs which are all by contraposi-} \underline{tion. In all of them, the negation of \CA\  is assumed}, and the cut involved is as provided in Remark~\ref{CArem}, which allows us to use the function $f$ in \eqref{MainCounterEx}.

\begin{proof}[Proof of $\IVT\implies\CA$.] The function $f$ given by \eqref{MainCounterEx} is continuous on $[a,b]$, with $f(a)=-1<0<1=f(b)$, yet $f(c)\in\{-1,1\}$ for every\linebreak $c\in[a,b]=(A\cap[a,b])\cup(B\cap[a,b])$, or that $f$ is never zero on $[a,b]$. Thus, \IVT\  is false.
\end{proof}

\begin{proof}[Proof of $\DT\implies\CA$.]
The functions $x\mapsto -x$ and $x\mapsto x$ are differentiable on $\R$ with derivatives $-1$ and $1$, respectively. Thus, the function given by \[
F(x):=\begin{cases}
-x, & \text{if } x\in A\cap[a,b],\\
x, & \text{if } x\in B\cap[a,b],
\end{cases}
\]
has derivative $f$ given in \eqref{MainCounterEx}, which is continuous on $[a,b]$, with $F'(a)=f(a)=-1<0<1=f(b)=F'(b)$, yet $F'(c)\in\{-1,1\}$ for every $c\in[a,b]=(A\cap[a,b])\cup(B\cap[a,b])$, or that $F'$ is never zero on $[a,b]$, making \DT\   false.
\end{proof}

\begin{proof}[Proof of $\CFT\implies\CA$.] The function $f$ given by \eqref{MainCounterEx} has zero derivative on $[a,b]$, but $f(a)=-1\neq 1=f(b)$, so $f$ is not constant on $[a,b]$, making \CFT\  false.
\end{proof}

\begin{proof}[Proof of $\iiFTC\implies\CA$.] The function $f$ given by \eqref{MainCounterEx} is continuous and has zero derivative on $[a,b]$, so $f$ is an antiderivative of the zero function on $[a,b]$, with $\int_a^b0=0(b-a)=0\neq 2=f(b)-f(a)$, so \iiFTC\  is false.
\end{proof}

Given $S\sub\R$ and $c\in \R$, the assertion that any open interval $(a,b)$ with $c\in (a,b)$ intersects $S$ is equivalent to the assertion that for any $\delta>0$, there exists $x\in\R$ such that $|x-c|<\delta$ and $x\in S$. (For necessity, the assumption applies to the interval $(c-\delta,c+\delta)$, while for sufficiency, use the fact that if $\delta:=\min\{b-c,c-a\}$, then $(c-\delta,c+\delta)\sub(a,b)$.) In such a case, we say that $c$ is a \emph{point of closure} of $S$. If every point of closure of $S$ is an element of $S$, then $S$ is a closed set. All closed intervals are closed sets.

\begin{proposition}[Pasting Lemma] If $A$ and $B$ are closed subsets of $\R$ with $A\cap B=\{c\}$, if $F_1$ is continuous on $A$, if $F_2$ is continuous on $B$ and if $F_1(c)=F_2(c)$, then the function $f:A\cup B\into\R$ defined by
\begin{eqnarray}
    f(x)=\begin{cases}
        F_1(x), &x\in A,\\
        F_2(x), &x\in B,
    \end{cases}\nonumber
\end{eqnarray}
is continuous on $A\cup B$.
\end{proposition}
\begin{proof}[Proof (Sketch).] Given $c_0\in A\cup B$ and given $\varepsilon>0$, routine epsilon-delta arguments may be used to prove the continuity of $f$ at $c_0$ for the cases\linebreak $c_0\in A\setdiff B$ or $c_0\in B\setdiff A$ using the continuity of $F_1$ on $A$ or the continuity of $F_2$ on $B$, respectively. When $c_0=c$, the continuity of $F_1$ on $A$ and of $F_2$ on $B$ may be used to produce a $\delta_1>0$ and $\delta_2>0$, respectively, satisfying certain properties, and the epsilon-delta proof for the continuity of $f$ at $c$ may be obtained by using $\delta:=\min\{\delta_1,\delta_2\}$, and using the fact that $F_1(c)=f(c)=F_2(c)$.
\end{proof}

\begin{proposition} If $(A,B)$ is a cut of $\R$ such that no cut point of $(A,B)$ is in $A$, then there exists a strictly increasing sequence $(a_n)$ in $A$ such that $A=(-\infty,a_1]\cup\bigcup_{n=1}^\infty[a_n,a_{n+1}]$.
\end{proposition}
\begin{proof} Since $A$ is nonempty, there exists $a_1\in A$. Suppose that for some $n\in\N$, elements of $A$ have been chosen, ordered as $a_1<a_2<\cdots<a_n$. If $(a_n,\infty)\cap A=\emptyset$, then $(a_n,\infty)\sub A^c=B$. From $(a_n,\infty)\sub B$, by taking complements, $A=B^c\sub(-\infty,a_n]$. Since $a_n\in A$, by the definition of cut, $b\in B$ implies $a_n<b$, which implies $a_n\leq b$, so $B\sub[a_n,\infty)$. We have thus proven that $a_n$ is a cut point of $(A,B)$ that is in $A$, contradicting an assumption about $A$. Hence, $(a_n,\infty)\cap A\neq\emptyset$, or that there exists $a_{n+1}\in A$ such that $a_n<a_{n+1}$, completing the induction.
\end{proof}

\begin{proof}[Proof of $\AP\&\BVT\implies\CA$.] We prove the equivalent form: \AP\  and the negation of \CA\  imply the negation of \BVT. Suppose $(A,B)$ is a cut of $\R$ that does not have a cut point, and let $(a_n)$ be a strictly increasing sequence in $A$ such that $A=(-\infty,a_1]\cup\bigcup_{n=1}^\infty[a_n,a_{n+1}]$. For each $n\in\N$, the linear function $F_n:[a_n,a_{n+1}]\into\R$ defined by $F_n(x):=n+\dfrac{x-a_n}{a_{n+1}-a_n}$ is continuous on $[a_n,a_{n+1}]$. Also, $F_n(a_{n+1})=n+1=F_{n+1}(a_{n+1})$ for all $n\in\N$. The constant function $F_0:(-\infty,a_1]\into\R$ defined by $F_0(x):=1$ is continuous on $(-\infty,a_1]$, and has the property $F_0(a_1)=1=F_1(a_1)$. We define $f:\R\into\R$ by \[
f(x):=\begin{cases}
F_0(x), & \text{if } x\in(-\infty,a_1],\\[0.5em]
F_n(x), & \text{if } x\in[a_n,a_{n+1}]\text{ for some }n\in\N,\\[1em]
0, & \text{if } x\in B.
\end{cases}
\]
By routine induction, we may use the Pasting Lemma to prove that $f$ is continuous on $(-\infty,a_1]\cup\bigcup_{n=1}^{m-1}[a_n,a_{n+1}]=(-\infty,a_m]$ for all $m\in\N$ with $m\geq 2$. Given $a\in A$, because $A=(-\infty,a_1]\cup\bigcup_{n=1}^\infty[a_n,a_{n+1}]$, there exists $N\in\N$ such that $a\in (-\infty,a_1]\cup\bigcup_{n=1}^N[a_n,a_{n+1}]=(-\infty,a_N]\sub(-\infty,a_{N+1})$. The continuity of $f$ on $(-\infty,a_{N+1}]$ implies the continuity of $f$ on $(-\infty,a_{N+1})\sub(-\infty,a_{N+1}]$. Using the fact that $(-\infty,a_{N+1})$ is an open set, a routine epsilon-delta argument may be used to prove that $f$ is continuous at $a$. Since the zero function is continuous on $B$, by Proposition~\ref{OpenPastingProp}, $f$ is continuous on $A\cup B=\R$, and hence on $[a_1,b]$ for any $b\in B$. Since $(a_n)$ is a sequence in $A$, we have $a_n<b$ for all $n\in\N$, and since $(a_n)$ is strictly increasing, we further have $a_1< a_n<b$ for all $n\in\N$ with $n\geq 2$. Thus, $(a_n)$ is a sequence in $[a_1,b]$ for any $b\in B$. For any $M\in\R$, by \AP, there exists $K\in\N$ such that $K>M$. But for $a_K\in[a_K,a_{K+1}]\sub[a_1,b]$, we have $f(a_K)=K>M$, and this falsifies \BVT.
\end{proof}

With this final implication established, every terminal node in our earlier summary diagram has now been shown to lead back to \CA, thereby closing the circle of implications. The complete logical structure of our naive treatment, incorporating both the sufficiency proofs of Sections~\ref{AxiomSec}--\ref{iiFTCsec} and the necessity proofs of the present section, is illustrated below.

\begin{center}
\begin{tikzpicture}[
    scale=0.5,             
    transform shape,       
    double arrow/.style={
        double, 
        double distance=1.5pt, 
        draw=#1, 
        ->, 
        >={Latex[length=4pt, width=5pt, color=#1]}, 
        shorten >= 3pt, 
        shorten <= 3pt,
        rounded corners=10pt
    },
    vertex/.style={
        draw, 
        rectangle, 
        rounded corners=5pt, 
        minimum width=40pt, 
        minimum height=20pt
    },
    wide box/.style={
        draw=red,
        fill=red!5,
        rectangle,
        rounded corners=5pt,
        minimum width=25.0cm,
        minimum height=1.35cm
    }
]

    \node[wide box] (TopBox) at (7.50, 18.87) {};

    \node[vertex, draw=blue, fill=blue!10] (Star_top) at (-3.85, 18.87) {\Star};
    \node at (-2.85, 18.87) {$\equiv$};
    \node[vertex, draw=red, fill=red!10]  (CA_top)   at (-1.85, 18.87) {\CA};
    \node[vertex, draw=red, fill=red!10]  (ES)       at (5.05, 18.87) {\ES};
    \node[vertex, draw=red, fill=red!10]  (NIP)      at (11.95, 18.87) {\AP\&\NIP};
    \node[vertex, draw=red, fill=red!10]  (BWP)      at (18.85, 18.87) {\BWP};

    \node[vertex, draw=black, fill=darkgray!10]        (IVT)    at (-3.85, 13.585) {\IVT};
    \node[vertex, draw=green!60!black, fill=green!10] (iFTC)   at (-0.26, 13.585) {\AP\&\iFTC};
    \node[vertex, draw=red, fill=red!10]              (DIT)    at (4.04, 13.585) {\AP\&\DIT};
    \node[vertex, draw=black, fill=darkgray!10]        (AP_BVT) at (8.25, 13.585) {\AP\&\BVT};
    \node[vertex, draw=green!60!black, fill=green!10] (BVT)    at (11.76, 13.585) {\BVT};
    \node[vertex, draw=green!60!black, fill=green!10] (iiES)   at (14.65, 16.23) {\ES};
    \node[vertex, draw=green!60!black, fill=green!10] (EVT)    at (14.65, 13.585) {\EVT};
    \node[vertex, draw=green!60!black, fill=green!10] (SIFT)   at (17.6, 6.95) {\SIFT};
    \node[vertex, draw=black, fill=darkgray!10]        (DT)     at (18.85, 13.585) {\DT};

    \node[vertex, draw=green!60!black, fill=green!10] (IET)  at (4.04, 9.28) {\AP\&\IET};
    \node[vertex, draw=blue, fill=blue!10]           (Star) at (4.04, 4.98) {\Star};

    \coordinate (Star_lane) at (-2.41, 0);

    \draw[double arrow=red, shorten >= 0pt, shorten <= 0pt]            (TopBox.south -| DIT) -- (DIT);
    \draw[double arrow=green!60!black, shorten >= 0pt, shorten <= 0pt] (iFTC) -- (DIT);
    \draw[double arrow=green!60!black, shorten >= 0pt, shorten <= 0pt] (DIT) -- (AP_BVT);
    \draw[double arrow=red, shorten >= 0pt, shorten <= 0pt]            (DIT) -- (IET);
    \draw[double arrow=green!60!black, shorten >= 0pt, shorten <= 0pt] (IET) -| (iFTC);
    \draw[double arrow=red, shorten >= 0pt, shorten <= 0pt]            (TopBox.south -| Star_lane) |- (Star);
    \draw[double arrow=red, shorten >= 0pt, shorten <= 0pt]            (Star) -- (IET);

    \node[vertex, draw=green!60!black, fill=green!10] (RT)    at (14.65, 11.37) {\RT};
    \node[vertex, draw=green!60!black, fill=green!10] (eMVT)  at (14.65, 9.16) {\eMVT};
    \node[vertex, draw=green!60!black, fill=green!10] (MVT)   at (14.65, 6.95) {\MVT};
    \node[vertex, draw=black, fill=darkgray!10]        (iiFTC) at (14.65, 4.74) {\iiFTC};

    \node[vertex, draw=green!60!black, fill=green!10] (IFT) at (17.6, 4.86) {\IFT};
    \node[vertex, draw=black, fill=darkgray!10]        (CFT) at (17.6, 2.77) {\CFT};

    \draw[double arrow=green!60!black, <->, shorten >= 0pt, shorten <= 0pt] (CA_top) -- (ES);
    \draw[double arrow=green!60!black, <->, shorten >= 0pt, shorten <= 0pt] (ES) -- (NIP);
    \draw[double arrow=green!60!black, <->, shorten >= 0pt, shorten <= 0pt] (NIP) -- (BWP);

    \draw[double arrow=red, shorten >= 0pt, shorten <= 0pt]            (TopBox.south -| IVT) -- (IVT);
    \draw[double arrow=green!60!black, shorten >= 0pt, shorten <= 0pt] (TopBox.south -| iiES) -- (iiES);

    \draw[double arrow=green!60!black, shorten >= 0pt, shorten <= 0pt] (AP_BVT) -- (BVT);
    \draw[double arrow=green!60!black, shorten >= 0pt, shorten <= 0pt] (iiES) -- (EVT);
    \draw[double arrow=green!60!black, shorten >= 0pt, shorten <= 0pt] (BVT) -- (EVT);
    \draw[double arrow=green!60!black, shorten >= 0pt, shorten <= 0pt] (MVT) -- (SIFT);
    \draw[double arrow=green!60!black, shorten >= 0pt, shorten <= 0pt] (MVT) -- (iiFTC);

    \draw[double arrow=green!60!black, shorten >= 0pt, shorten <= 0pt] (SIFT) -- (IFT);
    \draw[double arrow=green!60!black, shorten >= 0pt, shorten <= 0pt] (IFT) -- (CFT);

    \draw[double arrow=green!60!black, shorten <= 0pt, shorten >= 0pt] (EVT) -- (RT);
    \draw[double arrow=green!60!black, shorten <= 0pt, shorten >= 0pt] (EVT) -- (DT);
    \draw[double arrow=green!60!black, shorten <= 0pt, shorten >= 0pt] (RT) -- (eMVT);
    \draw[double arrow=green!60!black, shorten <= 0pt, shorten >= 0pt] (eMVT) -- (MVT);

    \node[wide box] (BotBox) at (7.50, 0) {};
    \node[vertex, draw=red, fill=red!10] (CA_bot) at (7.50, 0) {\CA};

    \draw[double arrow=black, shorten >= 0pt, shorten <= 0pt] (IVT) -- (BotBox.north -| IVT);
    \draw[double arrow=black, shorten >= 0pt, shorten <= 0pt] (AP_BVT) -- (BotBox.north -| AP_BVT);
    \draw[double arrow=black, shorten <= 0pt, shorten >= 0pt] (iiFTC) -- (BotBox.north -| iiFTC);
    \draw[double arrow=black, shorten >= 0pt, shorten <= 0pt] (CFT) -- (BotBox.north -| CFT);
    \draw[double arrow=black, shorten <= 0pt, shorten >= 0pt] (DT) -- (BotBox.north -| DT);

\end{tikzpicture}
\end{center}

An important note which we now make is regarding the clarification on why \AP\&\BVT\ is treated as a terminal node and why a separate proof that $\AP\&\BVT$ implies $\CA$ was indispensable, even though in the summary diagram of Section~\ref{RealSumSec} the node \AP\&\BVT\ has an outgoing arrow to \BVT\ that eventually connects to \EVT, \DT, \iiFTC, and \CFT. The crucial point is that the implication into \EVT\ is not ``$\BVT$ implies $\EVT$'' alone, and actually has a conjunction in the hypothesis: ``$\ES\&\BVT$ implies $\EVT$,'' which was proven in Section~\ref{GeomPrinceSec}. Consequently, tracing a path from \AP\&\BVT\ through \BVT\ and \EVT\ down to \CA\ requires invoking \ES\ along the way, which already presupposes the completeness of $\R$ and would therefore be circular as a proof that \AP\&\BVT\ suffices for \CA. Because no path from \AP\&\BVT\ to the lower terminal nodes avoids the auxiliary assumption \ES, the conjunction \AP\&\BVT\ is logically terminal in the network of implications in Section~\ref{RealSumSec} and requires its own proof as a sufficient condition for \CA.

Combining these proofs of necessity of $\CA$ with the implications summarized in Section~\ref{RealSumSec}, we conclude that each of \IVT, \DT, \CFT, \iiFTC, and \AP\&\BVT\ is not merely a consequence of completeness, but is itself equivalent to the completeness of $\R$ (in conjunction with \AP\ for some of these statements). In this sense, the basic theorems of calculus capture the full strength of the completeness axiom, confirming that no weaker assumption on the real number system suffices to sustain elementary real function theory.

\section{Remarks and Credits}

We conclude this exposition with a series of historical, pedagogical, and bibliographical notes that situate our choices of definitions, axioms, and proof strategies within the broader literature on elementary real analysis.

\begin{enumerate}
\item\emph{Compactness arguments versus circles of implications.} The part of our naive development of real function theory that is in Sections~\ref{NestedSec}--\ref{SeqSec} is a response to \cite[Section~2]{tho07}. Perhaps the motivation rests on the thought that conceivably, every one of our real analysis principles is provable directly from a completeness axiom using the \emph{compactness argument} that corresponds to that axiom. The term ``compactness,'' as explained in \cite[p.~472]{tho07}, may be aptly used because the argument relies on the (topological) compactness of a closed and bounded interval.  We decided to retain the term compactness even if compact sets were not included in this exposition. The prospect of directly proving every real analysis principle  using each of the four compactness arguments seems dreadful, though. In the literature, we find a unified approach only being made for a few theorems, such as in \cite[pp.~59--61]{bro96} where the Bolzano-Weierstrass property was used to prove only the Extreme Value, Intermediate Value and Uniform Continuity Theorems. Another is in \cite[pp.~A-8,A-9]{sal07} where supremum arguments are used to prove only the Bounded Value and Intermediate Value Theorems. The opposite approach, which is of not using one compactness argument for everything, and instead, uses ``natural'' implications that are said to make the circles of implications expediently proven, is the essence of \cite{rie01}. In this exposition, we combined both approaches. The multiple proofs each using a different compactness argument, was done only for ``seed'' theorems, and the rest were proven according to the philosophy in \cite{rie01}.

\item\emph{On our selection of completeness axioms.} The proofs of the implications in Section~\ref{AxiomSec} may be said to be folklore in real analysis. No originality is claimed (throughout this exposition), and the proofs we presented may be considered as minor modifications of those in the references cited in this item, and these references are to the best of our knowledge. There are most likely other classics we have missed. Our proof that $\AP\&\NIP$ implies $\CA$ and that $\ES$ implies $\BWP$ is based on \cite[Section~1]{rie01}, while our proof that $\CA$ implies $\ES$ is loosely based on \cite[p.~100]{tei13}. The proof of $\BWP$ implies $\NIP$ is based on \cite[pp.~64--65]{kas71}. This proof is elegant, especially the part by contraposition involving the limit of the lengths of the intervals. It hints at the use of the monotone convergence principle for sequences, which we did not use here, but a small modification allows the use of $\BWP$. There is also a proof of the equivalence of four possible completeness axioms outlined in \cite[p.~2]{tay65}, which, instead of $\BWP$, makes use of $\AP$ and the Cauchy Completeness Criterion, which we did not use here to spare us of the cumbersome characterizations of cluster points, such as limits superior.

\item\emph{On the statement of \iFTC.} We maintain that the statement of $\iFTC$ as a conditional statement is such that the hypothesis is the continuity of $f$ on $[a,b]$ and the conclusion is that the rule of assignment $x\mapsto\int_a^xf$ is an antiderivative of $f$ on $[a,b]$. This conclusion implies that the rule of assignment $x\mapsto\int_a^xf$ is a function $[a,b]\into\R$. Using, for instance, \cite[Appendix~A2]{tei13}, if $\CA$ is false and $\AP$ is true, then there is an interval $[a,b]$ and a function $f$ continuous on $[a,b]$ that does not have an integral over $[a,b]$, so a function $[a,b]\into\R$ corresponding to the integral does not exist. We chose, however, to make $\iFTC$ a part of the chain of theorems in Section~\ref{FTCsec} in order to exhibit some natural implications that are also based on the circle of real analysis principles in \cite[pp.~18--20]{rie01}. The apparent need for a ``modified'' $\iFTC$ being put forth in \cite[pp.~267,273,280]{dev14} is about problems arising from a different statement: if $f$ is continuous on $[a,b]$ and if the rule of assignment $x\mapsto\int_a^xf$ is a function $[a,b]\into\R$, then $x\mapsto\int_a^xf$ is an antiderivative of $f$ on $[a,b]$.

\item\emph{On being mean to the Mean Value Theorem.} We find in \cite{ber67,coh67,tuc97} some interesting material challenging the Mean Value Theorem as a pillar of real analysis. We have abandoned this tempting simplification, and brought favor back to the \MVT, praised as the ``Fundamental Theorem of Differential Calculus'' \cite[p.~174]{bar11}. (Mentioning \cite{bar11} here, however, does not bring with it support for the use of tagged partitions in teaching real analysis, about which the author feels some reluctance.) Back to $\MVT$, we find that it also fits nicely the scheme of implications in \cite{rie01}, particularly in the ``Third Circle'' \cite[pp.~11--15]{rie01}.

    \item \emph{On being mean to uniform continuity.} A proof of \DIT\  directly from definitions that does not make use of uniform continuity may be found in \cite[p.~139]{cla19}. Even if the method used in there is not a supremum argument, the proof is very suggestive that a supremum argument is plausible, and is hence implemented in this exposition. On the other hand, a proof of \DIT\  directly from definitions that also does not make use of uniform continuity but makes use of a nested interval argument may be found in \cite[p.~34]{oxt80}. In our proof that \IET\ implies \iFTC, even if the main mechanism is the limit of the difference quotient of the integral function, bypassing uniform continuity has been right under our noses the entire time: continuity itself guarantees a local form of uniform continuity, and in computing the limit of said difference quotient, we only need this local version. Thus, there is no need to use compactness to make the uniform continuity hold globally on the interval when later only local uniform continuity is needed.
    \item\emph{Forms and consequences of the Increasing Function Theorem.} Our proof in Section~\ref{GeomPrinceSec} of $\SIFT$ implies $\IFT$ is from the one-page paper \cite{ber67}. A proof of the converse is in \cite[p.~233]{tuc97}, which also includes a proof that $\IFT$ implies the Mean Value Inequality, which, even if not proven there, implies \MVT. The implication that $\MVT$ implies $\SIFT$ is standard calculus material. See, for instance, \cite[pp.~161--162]{sal07}. Thus, at this level, we already see the equivalence of \SIFT, \IFT\  and \MVT. We also find in \cite[p.~233]{tuc97} a proof that the Mean Value Inequality (or alternatively, $\MVT$) implies $\CFT$, while the one-page paper \cite{pow63} proves $\CFT$ using an ``infimum argument.''
    \item\emph{Other standard calculus proofs.} Aside from the proof of $\MVT$ implies $\SIFT$ mentioned in the previous item, we also cite here some standard calculus proofs, such as that of $\ES\&\BVT$ implies $\EVT$ \cite[p.~A-9]{sal07}, of $\RT$ implies $\eMVT$ \cite[pp.~557-558]{sal07} and of $\MVT$ implies $\iiFTC$ \cite[p.~110]{pro91}. That $\eMVT$ implies $\MVT$ perhaps needs no citation. Finally, our proofs of $\DT$ and $\RT$ as necessary conditions for $\EVT$ may be viewed as remakes, in the interest of expediency, of folklore proofs, for which no originality is being claimed, but to the best of our knowledge, do not appear as such in the usual references. 
    \item\emph{All roads lead to Dedekind.} Or more precisely, they all lead to the cuts of ordered fields named after him. Our proofs that $\IVT$ implies $\CA$ and that $\CFT$ implies $\CA$ in Section~\ref{DedekindSec} that make use of the function $f$ in Example~\ref{MainCounterExEx} are from \cite[pp.~399,400]{pro13}. The proof that $\iiFTC$ implies $\CA$ from \cite[p.~280]{dev14} was easily modified so that the same function $f$ is used, and the proof that $\DT$ implies $\CA$ from \cite[p.~102]{tei13} was also easily modified so that the derivative of the function in question is made to be $f$. By doing these modifications, we believe we have better organized the exposition. Our proof that $\AP\&\BVT$ implies $\CA$ was adapted from \cite[p.~108]{tei13}. In it, the condition of ``countable cofinality,'' which is a necessary condition for $\AP$, was used. Since countable cofinality and $\BVT$ imply $\CA$, we find that $\AP\&\BVT$ implies $\CA$. This valid argument is part of the reason the naivete of our treatment of real function theory is such that we did not need to involve countable cofinality, and hence the proof in \cite[p.~108]{tei13} was adapted here for the case when the countably cofinal set used is $\N$.  
    
    In doing this exposition, we had the opinion that $\CA$ lends itself to being the completeness axiom easiest to prove as terminal because the counterexamples needed were conveniently defined out of a Dedekind cut.
\end{enumerate}

\section*{Dedication}
This exposition is dedicated to the memory of M. U. Oreta.


\begin{thebibliography}{99}

\bibitem{bar11}
R.G. Bartle, D.R. Sherbert, Introduction to Real Analysis, fourth ed., John Wiley \& Sons, Hoboken, N.J., 2011.

\bibitem{ber67}
L. Bers, On avoiding the mean value theorem, Amer. Math. Mon. 74 (1967) 583.

\bibitem{bro96}
A. Browder, Mathematical Analysis: An Introduction, Undergrad. Texts Math., Springer-Verlag, New York, 1996.

\bibitem{can24}
R.R.S. Cantuba, Littlewood's principles in reverse real analysis, Real Anal. Exch. 49 (1) (2024) 111--122.

\bibitem{cla19}
P.L. Clark, The instructor's guide to real induction, Math. Mag. 92 (2) (2019) 136--150.

\bibitem{coh67}
L.W. Cohen, On being mean to the mean value theorem, Amer. Math. Mon. 74 (1967) 581--582.

\bibitem{dev14}
M. Deveau, H. Teismann, $72+42$: characterizations of the completeness and Archimedean properties of ordered fields, Real Anal. Exch. 39 (2) (2013/14) 261--304.

\bibitem{kas71}
R.H. Kasriel, Undergraduate Topology, W.B. Saunders Company, Philadelphia, 1971.

\bibitem{loo74}
L. Loomis, Calculus, Addison-Wesley Publishing Company, Reading, Mass., 1974.

\bibitem{olm73}
J.M.H. Olmsted, Riemann integration in ordered fields, Two-Year Coll. Math. J. 4 (2) (1973) 34--40.

\bibitem{oxt80}
J.C. Oxtoby, Measure and Category: A Survey of the Analogies Between Topological and Measure Spaces, second ed., Grad. Texts Math., vol. 2, Springer, Cham, 1980.

\bibitem{pow63}
M. Powderly, A simple proof of a basic theorem of the calculus, Amer. Math. Mon. 70 (1963) 544.

\bibitem{pro13}
J. Propp, Real analysis in reverse, Amer. Math. Mon. 120 (5) (2013) 392--408.

\bibitem{pro91}
M.H. Protter, C.B. Morrey Jr., A First Course in Real Analysis, second ed., Undergrad. Texts Math., Springer-Verlag, New York, 1991.

\bibitem{rie01}
O. Riemenschneider, 37 elementare axiomatische Charakterisierungen des reellen Zahlk\"orpers, Mitt. Math. Ges. Hamburg 20 (2001) 71--95 (English translation available at \url{https://www.math.uni-hamburg.de/home/riemenschneider/axioms.pdf}).

\bibitem{roy10}
H. Royden, P.M. Fitzpatrick, Real Analysis, fourth ed., Prentice Hall, New York, 2010.

\bibitem{sal07}
S.L. Salas, E. Hille, G.J. Etgen, Calculus: One and Several Variables, tenth ed., John Wiley \& Sons, Hoboken, N.J., 2007.

\bibitem{sus13}
L. Susskind, G. Hrabovsky, The Theoretical Minimum: What You Need to Know to Start Doing Physics, Penguin eBooks, 2013.

\bibitem{tay65}
A.E. Taylor, General Theory of Functions and Integration, Blaisdell Publishing Company, New York, 1965.

\bibitem{tei13}
H. Teismann, Toward a more complete list of completeness axioms, Amer. Math. Mon. 120 (2) (2013) 99--115.

\bibitem{tho07}
B.S. Thomson, Rethinking the elementary real analysis course, Amer. Math. Mon. 114 (6) (2007) 469--490.

\bibitem{tuc97}
T.W. Tucker, Rethinking rigor in calculus: the role of the mean value theorem, Amer. Math. Mon. 104 (3) (1997) 231--240.
\end{thebibliography}
\end{document}